\documentclass[12pt,a4paper]{article}
\usepackage[utf8]{inputenc}
\usepackage[english]{babel}
\usepackage{xcolor}

\usepackage{amssymb,mathtools,amsmath,amsthm,stmaryrd}
\usepackage[all,cmtip,2cell]{xy} 

\usepackage[draft=false, hidelinks, linkcolor = blue]{hyperref}
\usepackage{cleveref}
\usepackage{tensor}
\usepackage{url}
\usepackage{bookmark}

\usepackage{graphicx}
\usepackage{lmodern}
\usepackage{enumitem}
\usepackage{array}

\usepackage[left=2.5cm,right=2.5cm,top=2.5cm,bottom=2.5cm]{geometry}
\def\defthm#1#2#3#4{
  \newtheorem{#1}[theorem]{#3}
  \newtheorem*{#1*}{#3}
  \newtheorem{#2}[theorem]{#4}
  \newtheorem*{#2*}{#4}
  \crefname{#1}{#3}{#4}
  \crefname{#2}{#4}{#4}  
}

\newtheoremstyle{mythm}%
{10pt}
{}
{\itshape}
{}
{\bf}
{.}
{.5em}
{}%

\newtheoremstyle{mydef}%
{10pt}
{3pt}
{}
{}
{\bf}
{.}
{.5em}
{}%

\newtheoremstyle{myrmk}%
{10pt}
{3pt}
{}
{}
{\bf}
{.}
{.5em}
{}%

\theoremstyle{mythm}
\newtheorem{theorem}{Theorem}[section]
\newtheorem*{theorem*}{Theorem}

\defthm{corollary}{corollaries}{Corollary}{Corollaries}
\defthm{lemma}{lemmata}{Lemma}{Lemmata}
\defthm{proposition}{propositions}{Proposition}{Propositions}
\defthm{axiom}{axioms}{Axiom}{Axioms}
\defthm{propdef}{props/defs}{Proposition/Definition}{Prositions/Definitions}
\defthm{defcor}{defscors}{Corollary/Definition}{Corollaries/Definitions}
\defthm{conjecture}{conjectures}{Conjecture}{Conjectures}

\theoremstyle{mydef}
\defthm{definition}{definitions}{Definition}{Definitions}

\theoremstyle{myrmk}
\defthm{acknowledgment}{acknowledgments}{Acknowledgment}{Acknowledgments}
\defthm{remark}{remarks}{Remark}{Remarks}
\defthm{motivation}{motivations}{Motivation}{Motivations}
\defthm{example}{examples}{Example}{Examples}
\defthm{question}{questions}{Question}{Questions}
\defthm{observation}{observations}{Observation}{Observations}
\defthm{claim}{claims}{Claim}{Claims}
\defthm{notation}{notations}{Notation}{Notations}

\newtheorem*{replemmax}{\reptitle}
 {\end{replemmax}}

\newtheorem*{repthmx}{\reptitle}
 {\end{repthmx}}

\newtheorem*{repcorx}{\reptitle}
 {\end{repcorx}}

\crefname{section}{Section}{Sections}
\crefname{theorem}{Theorem}{Theorems}
\numberwithin{equation}{section}

\makeatletter
\renewenvironment{proof}[1][\proofname] {\par\pushQED{\qed}\normalfont\topsep6\p@\@plus6\p@\relax\trivlist\item[\hskip\labelsep\bf#1\@addpunct{.}]\ignorespaces}{\popQED\endtrivlist\@endpefalse}
\makeatother

\newcommand{\blank}{\mbox{\hspace{3pt}\underline{\ \ }\hspace{2pt}}}
\newcommand{\sprime}{^{\prime}}

\newcommand{\pbs}{\scalebox{1.5}{\rlap{$\cdot$}$\lrcorner$}}
\newcommand{\pos}{\rotatebox[origin=c]{180}{\pbs}}

\newcommand{\bBox}{\mathbin{\Box}}

\author{Raffael Stenzel\thanks{The author acknowledges the support of the Grant Agency of the Czech Republic under the
grant 19-00902S.}}
\renewcommand\footnotemark{}

\title{Bousfield-Segal spaces}

\begin{document}
\maketitle

\abstract{
This paper is a study of Bousfield-Segal spaces, a notion introduced by Julie Bergner drawing on ideas about 
Eilenberg-Mac Lane objects due to Bousfield. In analogy to Rezk's Segal spaces, 
they are defined in such a way that Bousfield-Segal spaces naturally come equipped with a 
homotopy-coherent fraction operation in place of a composition. 

In this paper we show that Bergner's model structure for Bousfield-Segal spaces in fact can be obtained from 
the model structure for Segal spaces both as a localization and a colocalization. We thereby prove that Bousfield-Segal 
spaces really are Segal spaces, and that they characterize exactly those with invertible arrows.
We note that the complete Bousfield-Segal spaces are precisely the homotopically constant Segal spaces, and deduce that 
the associated model structure yields a model for both $\infty$-groupoids and Homotopy Type Theory.
}

\section{Introduction}
In \cite[6]{bergner2}, Julie Bergner introduced two model structures on the category $s\mathbf{S}$ of bisimplicial sets.
Characterized by their fibrant objects, these are the model structure $(s\mathbf{S},\mathrm{B})$ for Bousfield-Segal 
spaces and the model structure $(s\mathbf{S},\mathrm{CB})$ for complete Bousfield-Segal spaces. A reduced and pointed  
version of Bousfield-Segal spaces itself originated in Bousfield's work \cite{bousfieldnotes} under the name ``very 
special bisimplicial sets of type 1'' (representing what he calls Eilenberg-Mac Lane objects of type 1 in the homotopy 
category of spaces). A pointed version of complete Bousfield-Segal spaces appeared in the same 
unpublished note under the name ``very special bisimplicial sets of type 0''.

Both notions were defined in the last 
section of Bergner's paper, proposing model structures whose fibrant objects are to be thought of as 
$\infty$-groupoidal ``Segal-like'' spaces.
The primary topic of this paper is to study these two model structures and relate them to Rezk's model structures for 
Segal spaces and complete Segal spaces as introduced in \cite{rezk}.

In essence, the relation between Rezk's notions and Bergner's notions is the relation between formal multiplication 
and formal division. In 1-dimensional algebra, taking fractions in a group yields an operation
$(g,h)\mapsto g/h$ on the underlying set of the group which is defined by the multiplication $gh^{-1}$. This operation 
satisfies various properties which may be abstractly axiomatized so as to define a formal fraction operation on any set 
without the assumption of a multiplication in the first place (see \cite[1.4]{bousfieldnotes} or 
\cite[1.3]{hallgrpthy}).

To this effect, Bousfield-Segal spaces as defined in the work of Bergner are Reedy fibrant bisimplicial sets $X$ which 
come equipped with a homotopy-coherent and many-sorted fraction operation. If we denote the space of edges between two 
vertices $x,y\in X_{00}$ in $X$ by $X_1(x,y)$, that is a family of arrows
\[\blank /\blank\colon X_1(x,z)\times X_1(x,y)\rightarrow X_1(y,z)\]
induced by associated \emph{Bousfield maps}, in a very similar way as Segal spaces are 
Reedy fibrant bisimplicial sets $Y$ equipped with a composition operation
\[\circ\colon Y_1(y,z)\times Y_1(x,y)\rightarrow Y_1(x,z)\]
induced by associated Segal maps.
It is a classical result of ordinary group theory that fraction operations (subject to suitable axioms) and group 
structures yield equivalent data on any given set. Accordingly, we will exhibit Bousfield-Segal spaces as exactly those 
Segal spaces in which every edge is an equivalence.

To explain how the latter arise as the fibrant objects in a left Bousfield localization of Rezk's model 
structure for Segal spaces in a very natural way, let us recall that the category $\mathbf{Gpd}$ of (small) groupoids 
arises as a localization of the category $\mathbf{Cat}$ of (small) categories. If by $I[1]$ we denote the free groupoid 
generated by the walking arrow $[1]$, then $\mathbf{Gpd}$ is the localization of
$\mathbf{Cat}$ at the inclusion $[1]\rightarrow I[1]$. The model structure for 
Kan complexes can be obtained similarly as the left Bousfield localisation of the model structure for
quasi-categories at the inclusion $\Delta^1\rightarrow N(I[1])$, and, indeed, Kan complexes are exactly the
quasi-categories with invertible edges. Analogously, we will see that the model structure for Bousfield-Segal spaces is 
the left Bousfield localization of the model structure for Segal spaces at a canonical map induced by the inclusion
$[1]\rightarrow I[1]$.

Joyal and Tierney have shown in \cite[Theorem 4.11]{jtqcatvsss} that the model structure $(s\mathbf{S},\mathrm{CS})$ for 
complete Segal spaces is a model for $(\infty,1)$-category theory equivalent to the model category for quasi-categories. 
It hence follows that the model structure $(s\mathbf{S},\text{CB})$ for complete Bousfield-Segal spaces is a model for
$\infty$-groupoids equivalent to the one associated to Kan complexes, as stated in \cite[Theorem 6.12]{bergner2}.
We will furthermore see that $(s\mathbf{S},\text{CB})$ is right proper and supports a model of Homotopy 
Type Theory with univalent universes in the sense of \cite{shulmaninv}, using that Bousfield-Segal spaces are complete 
if and only if they are homotopically constant. 

Therefore, Section~\ref{secreedy} recalls the Reedy model structure $(s\mathbf{S},R_v)$ on bisimplicial sets 
and some of its associated Joyal-Tierney calculus from \cite[Section 2]{jtqcatvsss}.
Section~\ref{secbs} introduces Bousfield-Segal spaces in the sense of \cite{bergner2}. Here, we explain how every 
Bousfield-Segal space $X$ comes equipped with a contractible choice of fraction operations which induce an associated 
homotopy groupoid $\text{Ho}_B(X)$.

In Section~\ref{secb=bs} we will show that every Bousfield-Segal space is not just ``Segal-like'' but in fact a Segal 
space and that the model structure $(s\mathbf{S},\text{B})$ is a left Bousfield localization of
$(s\mathbf{S},\text{S})$. We will also see that the homotopy category $\text{Ho}(X)$ of a Bousfield-Segal space $X$ 
associated to it \emph{as a Segal space} (following \cite[5.5]{rezk}) is a groupoid and coincides with the construction 
$\text{Ho}_B(X)$. Hence, many of Rezk's results in \cite{rezk} and Joyal and Tierney's results in \cite{jtqcatvsss} 
carry over to the model structure for Bousfield-Segal spaces.

In Section~\ref{secinvedges} we use this to describe Bousfield-Segal spaces as the Segal spaces with invertible edges in 
a precise way. We furthermore define the core of a Segal space $X$ as the largest Bousfield-Segal space contained in $X$ 
and show that this construction exhibits $(s\mathbf{S},\text{B})$ as a homotopy colocalization of
$(s\mathbf{S},\text{S})$ as well.

In Section 6 we study complete Bousfield-Segal spaces and show that they are exactly the Reedy fibrant homotopically 
constant bisimplicial sets. The model structure for homotopically constant bisimplicial sets is contained in different 
classes of well understood model structures treated in the literature of \cite{rss}, \cite{duggersmallpres} and 
\cite{cisinski} respectively. It follows that the vertical projection, the horizontal projection and the diagonal 
functor all are part of Quillen equivalences between the model category $(s\mathbf{S},\mathrm{CB})$ for complete 
Bousfield-Segal spaces and the model category $(\mathbf{S},\mathrm{Kan})$ for Kan complexes. In Section~\ref{seccbshott} 
we conclude that $(s\mathbf{S},\mathrm{CB})$ is a type theoretic model category with as many univalent fibrant universes 
as there are inaccessible cardinals. We give a direct proof that $(s\mathbf{S},\mathrm{CB})$ is right proper using a 
symmetry argument.

\begin{acknowledgments*}
The author would like to thank his advisor Nicola Gambino for proposing the study of Bergner's 
work on invertible Segal spaces with a view towards Homotopy Type Theory, as well as for the many subsequent discussions 
and his continuous feedback. The author is grateful to John Bourke, Christian Sattler and an anonymous referee for 
helpful comments, and to the organizers of the YaMCatS series, the MURI HoTT Meeting 2017, and the PSSL 2018 in Brno, 
respectively, for giving him the opportunity to present this material. The author would also like to thank Julie Bergner 
for sharing her personal notes on Bousfield's original work on very special bisimplicial sets. The majority of the 
contents of this paper were carried out as part of the author's PhD thesis at the University of Leeds, supported by a 
Faculty Award of the 110 Anniversary Research Scholarship. The present version is written with support of the Grant 
Agency of the Czech Republic under the grant 19-00902S.
\end{acknowledgments*}

\section{Preliminaries on bisimplicial sets}\label{secreedy}

A bisimplicial set $X\in s\mathbf{S}$ can be understood as a functor
$X\colon\Delta^{op}\times\Delta^{op}\rightarrow\text{Set}$, and whenever done so, will be denoted by 
$X_{\bullet\bullet}$ to highlight its two 
components. Taking its exponential transpose associated to the product with $\Delta^{op}$ on the left or on the right 
yields simplicial objects in the category $\mathbf{S}$ of simplicial sets, whose evaluation at an object
$[n]\in\Delta^{op}$ is the $n$-th column $X_{n}:=X_{n\bullet}$ and the $n$-th row $X_{\bullet n}$ respectively.

\subsection{The box product and its adjoints}
To recall some constructions which are very convenient in describing the generating sets for
the model structures on bisimplicial sets we are interested in, we briefly summarise some constructions from 
\cite[Section 2]{jtqcatvsss}.

By left Kan extension of the Yoneda embedding $y\colon\Delta\times\Delta\rightarrow s\mathbf{S}$  along the 
product of Yoneda embeddings
$y\times y\colon\Delta\times\Delta\rightarrow\mathbf{S}\times\mathbf{S}$
one obtains a bicontinuous functor
$\blank\bBox\blank\colon \mathbf{S}\times \mathbf{S}\rightarrow s\mathbf{S}$, often called the
\emph{box product}. The box product is divisible on both sides, i.e.\ gives rise to adjoint pairs 
\[\xymatrix{
A\bBox\blank\colon\mathbf{S}\ar@/^.5pc/[r]\ar@{}[r]|{\bot} & s\mathbf{S}\colon A\setminus\blank\ar@/^.5pc/[l]
}\]
and 
\[\xymatrix{
\blank\bBox B\colon\mathbf{S}\ar@/^.5pc/[r]\ar@{}[r]|{\bot} & s\mathbf{S}\colon\blank/B\ar@/^.5pc/[l]
}\]
for all simplicial sets $A$ and $B$. 
In particular, for any bisimplicial set $X$, the simplicial set $\Delta^n\setminus X\cong X_n$ is the $n$-th 
column and $X/\Delta^n\cong X_{\bullet n}$ is the $n$-th row of $X$. Vice versa, for a given
$X\in s\mathbf{S}$, the induced functors
\[\xymatrix{
\blank\setminus X\colon\mathbf{S}^{op}\ar@<.5ex>[r] & \mathbf{S}\colon X/\blank\ar@<.5ex>[l]
}\]
form an adjoint pair with left adjoint $\blank\setminus X$. The box product induces a functor
\[\blank\bBox\sprime\blank\colon\mathbf{S}^{[1]}\times\mathbf{S}^{[1]}\rightarrow (s\mathbf{S})^{[1]}\]
on arrow-categories via the pushout-product construction. That means, it takes a pair of arrows
$u\colon A\rightarrow B$, $v\colon A\sprime\rightarrow B\sprime$ in $\mathbf{S}$ to the natural map $u\bBox\sprime v$ in 
the diagram
\[\xymatrix{
A\bBox A\sprime\ar[d]_{v\bBox A\sprime}\ar[r]^{A\bBox u}\ar@{}[dr]|(.7){\pos} & A\bBox B\sprime\ar@/^1pc/[ddr]^{v\bBox B\sprime}\ar[d] & \\
B\bBox A\sprime\ar[r]\ar@/_1pc/[drr]_{B\bBox u} & Q\ar[dr]|{u\bBox\sprime v} & \\
 & & B\bBox B\sprime.
 }\]
The functor $\blank\bBox\sprime\blank$ is divisible on both sides, too; the right adjoints to the functors
$f\bBox\sprime\blank$ and $\blank\bBox\sprime f$ for a given map $f\in s\mathbf{S}$ are denoted by
\[\langle f\setminus\blank\rangle,\langle\blank/f\rangle
\colon(s\mathbf{S})^{[1]}\rightarrow\mathbf{S}^{[1]}\]
respectively.
In the following, given arrows $u,v$ in $\mathbf{S}$ we write $u\pitchfork v$ to denote that $u$ has the left lifting 
property against $v$.

\begin{proposition}[{\cite[Proposition 2.1]{jtqcatvsss}}]\label{boxdividadj}
For any two maps $u,v\in\mathbf{S}$ and any map $f\in s\mathbf{S}$, we have
\[(u\bBox\sprime v)\pitchfork f\Longleftrightarrow u\pitchfork\langle f/v\rangle\Longleftrightarrow v\pitchfork\langle u\setminus f\rangle.\]
\end{proposition}\qed

\subsection{The vertical and horizontal Reedy model structures}
It is well known that the Reedy and injective model structures on $s\mathbf{S}$ coincide. This coincidence boils down to 
the fact that the simplex-category $\Delta$ is an elegant Reedy category as treated by Bergner and Rezk in
\cite[3]{brtheta} (in fact it is the archetype of such a Reedy category). We loosely follow the language and 
structure of \cite{jtqcatvsss} and call this model structure the \emph{vertical Reedy model} structure, denoted by
$R_v$. Its cofibrations are the (pointwise) monomorphisms, its weak equivalences the pointwise weak equivalences and its 
fibrations the maps with the right lifting property with respect to those cofibrations which are also weak equivalences. 
The classes of vertical Reedy cofibrations, weak equivalences and fibrations will be denoted by
$\mathcal{C}_v$, $\mathcal{W}_v$ and $\mathcal{F}_v$, respectively. We call the fibrations in the vertical Reedy model 
structure ``$v$-fibrations'' for short.

For $n\geq 0$ we denote by $\delta_n\colon\partial\Delta^n\hookrightarrow\Delta^n$ the boundary inclusion of the
$n$-simplex $\Delta^n\in\mathbf{S}$ and, for $0\leq i\leq n$, by $h_i^n\colon \Lambda_i^n\hookrightarrow\Delta^n$ the 
corresponding $i$-th horn inclusion. Recall that the set $\{\delta_n\mid n\geq 0\}$ of boundary inclusions generates the 
class of cofibrations and the set $\{h_i^n\mid 0\leq i\leq n, 1\leq n\}$ of horn inclusions generates the class of 
acyclic cofibrations in the Quillen model structure $(\mathbf{S},\text{Kan})$. These acyclic cofibrations are often referred to as the \emph{anodyne} maps, and their corresponding fibrations are the \emph{Kan fibrations}.

In terms of the general calculus of Reedy structures as presented for example in \cite[Section 5.2]{hovey}, the 
object $\partial\Delta^n\setminus X$ is the $n$-th matching object of $X$. Hence, by
\cite[Theorem 5.2.5]{hovey}, a map
$f\colon X\rightarrow Y$ in $(s\mathbf{S},R_v)$ is an (acyclic) $v$-fibration if and only if the associated maps
\[\langle\delta_n\setminus f\rangle\colon X_n\rightarrow Y_n\times_{(\partial\Delta^n\setminus Y)}(\partial\Delta^n\setminus X)\]
are (acyclic) Kan fibrations in $\mathbf{S}$.
One can show that the class of cofibrations $\mathcal{C}_v$ of $(s\mathbf{S},R_v)$ is generated by the set
\begin{align}\label{equreedygencofs}
\begin{gathered}
\mathcal{I}_v:=\{\delta_n\bBox\sprime\delta_m\colon(\Delta^n\bBox\partial\Delta^m)\cup_{\partial\Delta^n\bBox\partial\Delta^m}(\partial\Delta^n\bBox\Delta^m)\rightarrow(\Delta^n\bBox\Delta^m)\mid 0\leq m,n\},
\end{gathered}
\end{align}
and the class $\mathcal{W}_v\cap\mathcal{C}_v$ of acyclic cofibrations is generated by the set
\begin{align}\label{equreedygenaccofs}
\begin{gathered}
\mathcal{J}_v:=\{\delta_n\bBox\sprime h_i^m\colon(\Delta^n\bBox\Lambda_i^m)\cup_{\partial\Delta^n\bBox\Lambda_i^m}(\partial\Delta^n\bBox\Delta^m)\rightarrow(\Delta^n\bBox\Delta^m)\mid 0\leq n,1\leq m, 0\leq i\leq m\}.
\end{gathered}
\end{align}
It is easy to show that properness of $(\mathbf{S},\text{Kan})$ implies properness of $(s\mathbf{S},R_v)$, since every
$v$-(co)fibration is also a pointwise (co)fibration, and the Reedy weak equivalences are exactly the pointwise weak 
equivalences. 

Lastly, we recall that the vertical Reedy model structure is simplicially enriched as follows. The projection
$p_2\colon\Delta\times\Delta\rightarrow\Delta$ onto the second component and the corresponding inclusion
$\iota_2=([0],\text{id})\colon\Delta\rightarrow\Delta\times\Delta$ constitute an adjoint pair $p_2\dashv\iota_2$, and 
hence give rise to an adjoint pair
\[\xymatrix{
p_2^{\ast}\colon\mathbf{S}\ar@/^.5pc/[r]\ar@{}[r]|{\bot} & s\mathbf{S}\colon\iota_2^{\ast}\ar@/^.5pc/[l],
}\]
with $(p_2^{\ast}A)_n=A$ for all $n\geq 0$, and $\iota_2^{\ast}X=X_0$ the $0$-th column of $X$. The category
$s\mathbf{S}$ is a presheaf category and as such cartesian closed. Thus, for bisimplicial sets $X$ and $Y$ we obtain an 
exponential $Y^X\in s\mathbf{S}$, and thereby a simplicial enrichment of $s\mathbf{S}$ via
$\text{Hom}_2(X,Y):=\iota_2^{\ast}(Y^X)$.

\begin{proposition}[{\cite[Propositions 2.4 and 2.6]{jtqcatvsss}}]\label{prophom2}
The simplicial enrichment $\text{Hom}_2$ on $s\mathbf{S}$ turns $(s\mathbf{S},R_v)$ into a simplicial model category.
\end{proposition}\qed

It follows that the Reedy model structure $(s\mathbf{S},R_v)$ is a simplicial and left proper cellular model category in 
the sense of \cite[12]{hirschhorn03}. Hence, by \cite[Theorem 4.1.1]{hirschhorn03}, given a set of maps
$A\subset s\mathbf{S}$, the left Bousfield localization of $(s\mathbf{S},R_v)$ at $A$ exists, and is simplicial, left 
proper and cellular again. We will denote this localization by $\mathcal{L}_A(s\mathbf{S},R_v)$. Its fibrant objects are 
exactly the $A$-local $v$-fibrant objects.

We hereby conclude the discussion of the vertical Reedy model structure.

The permutation $\sigma:=(p_2, p_1):\Delta\times\Delta\rightarrow\Delta\times\Delta$ induces 
an isomorphism $\sigma^{\ast}:s\mathbf{S}\rightarrow s\mathbf{S}$ which transports the vertical Reedy model 
structure into the \emph{horizontal Reedy model structure} $R_h$. Its class of cofibrations is given by
$\mathcal{C}_h=\{\text{monomorphisms in }s\mathbf{S}\}$, its weak equivalences are the rowwise weak homotopy 
equivalences,
\[\mathcal{W}_h=\{f\colon X\rightarrow Y\mid f_{\bullet n}\colon X_{\bullet n}\rightarrow Y_{\bullet n}\text{ is a weak homotopy equivalence for all }n\geq 0\}.\]
Its cofibrations and acyclic cofibrations are generated by the sets $\mathcal{I}_h=\mathcal{I}_v$ and
\[\mathcal{J}_h=\{h_i^n\bBox\sprime\delta_m\colon(\Delta^n\bBox\partial\Delta^m)\cup_{\Lambda_i^n\bBox\partial\Delta^m}(\Lambda_i^n\bBox\Delta^m)\rightarrow(\Delta^n\bBox\Delta^m)\mid 0\leq n, 1\leq m, 0\leq i\leq m\}\]
respectively. We denote its class of fibrations by $\mathcal{F}_h$.
In analogy to the pair $p_2^{\ast}\dashv\iota_2^{\ast}$, we have an adjunction
\begin{align}\label{defp_1}
\xymatrix{
p_1^{\ast}\colon\mathbf{S}\ar@/^.5pc/[r]\ar@{}[r]|{\bot} & s\mathbf{S}\colon\iota_1^{\ast}\ar@/^.5pc/[l]
}
\end{align}
with $(p_1^{\ast}A)_{\bullet n}=A$ for all $n\geq 0$, and $\iota_1^{\ast}X=X_{\bullet 0}$ the $0$-th row of $X$.

Joyal and Tierney show in \cite{jtqcatvsss} that $(s\mathbf{S},R_v)$ naturally comes equipped 
with two orthogonal projections, a Quillen right adjoint
$\iota_1^{\ast}\colon(s\mathbf{S},R_v)\rightarrow(\mathbf{S},\mathrm{Kan})$ on the one hand, and a mere right adjoint
$\iota_2^{\ast}\colon s\mathbf{S}\rightarrow\mathbf{S}$ on the other. In order to construct a homotopy 
theory of $(\infty,1)$-categories in $s\mathbf{S}$, they localize $(s\mathbf{S},R_v)$ at a suitable set 
of maps such that the horizontal projection $i_2^{\ast}\colon s\mathbf{S}\rightarrow\mathbf{S}$ 
becomes a Quillen right adjoint (and in fact part of a Quillen equivalence) to the Joyal model structure
$(\mathbf{S},\mathrm{QCat})$.
In the process, Segal spaces arise naturally as an intermediate step; their individual horizontal projections 
already yield quasi-categories objectwise.
 
In order to construct a homotopy theory of $\infty$-\emph{groupoids}, one can localize
$(s\mathbf{S},R_v)$ at a larger class of maps such that the horizontal projection
$\iota_2^{\ast}\colon s\mathbf{S}\rightarrow\mathbf{S}$ becomes a Quillen right adjoint (and in fact part of a Quillen 
equivalence) to the model structure for \emph{Kan complexes} $(\mathbf{S},\mathrm{Kan})$ as we will see in
Section~\ref{chcbs}. The corresponding intermediate objects, the Segal spaces with invertible arrows, are the 
subject of the following three sections.

\section{Bousfield-Segal spaces}\label{secbs}

In this section, we give the definition of Bousfield-Segal spaces as introduced in \cite[6]{bergner2} and describe an 
associated fraction operation they naturally come equipped with. The definition is phrased in terms of
\emph{Bousfield maps} associated to a bisimplicial set, and in order to motivate these let us first recall the 
definition of Segal spaces and their associated Segal maps.

Let $\mathrm{sp}_n\colon \mathrm{Sp}_n\hookrightarrow\Delta^n$ be the $n$-th spine-inclusion, i.e.\
\[\mathrm{Sp}_n=\bigcup_{i<n}\zeta_i[\Delta^1]\]
for $\zeta_i\colon[1]\rightarrow[n]$, $0\mapsto i$, $1\mapsto i+1$. Localizing $(s\mathbf{S},R_v)$ at 
the set of horizontally constant diagrams
\[\mathrm{S}:=\{p_1^{\ast}(\mathrm{sp}_n)\colon p_1^{\ast}(\mathrm{Sp}_n)\hookrightarrow p_1^{\ast}(\Delta^n)\mid 2\leq n\}\]
yields the left proper cellular simplicial model structure
$(s\mathbf{S},S):=\mathcal{L}_S(s\mathbf{S},R_v)$ whose fibrant objects are the \emph{Segal spaces} as 
defined in \cite[Section 4.1]{rezk} and \cite[Definition 3.1]{jtqcatvsss}. By construction, these are $v$-fibrant 
bisimplicial sets $X$ such that the maps
\begin{align}\label{equsecbs1}
(p_1^{\ast}(\mathrm{sp}_n))^{\ast}\colon\mathrm{Hom}_2(p_1^{\ast}(\Delta^n),X)\rightarrow\mathrm{Hom}_2(p_1^{\ast}(\mathrm{Sp}_n),X)
\end{align}
are weak homotopy equivalences for all $n\geq 2$.

For each $n\geq 2$, the map (\ref{equsecbs1}) is isomorphic to
$\mathrm{sp}_n\setminus X\colon\Delta^n\setminus X\rightarrow \mathrm{Sp}_n\setminus X$. Thus, its codomain is the 
pullback $X_1\times_{X_0}\dots\times_{X_0}X_1$ taken along the boundaries $d_0\setminus X$ and $d_1\setminus X$ 
successively. In the following, we denote this pullback by $X_1\times^S_{X_0}\dots\times^S_{X_0}X_1$. Then we define the \emph{Segal maps}
\begin{equation}\label{smap}
\xi_n\colon X_n\rightarrow X_1\times^S_{X_0}\dots\times^S_{X_0}X_1
\end{equation}
via $\xi_n:=\mathrm{sp}_n\setminus X$ for $n\geq 2$, such that Segal spaces are the $v$-fibrant bisimplicial sets whose 
associated Segal maps are acyclic fibrations. We can think of a Segal space $X$ as a simplicial collection of Kan 
complexes, where $X_0$ is its space of objects and $X_1$ is its space of edges. It comes equipped with a horizontal weak 
composition in the form of the diagram
$\xymatrix{X_1\times_{X_0}^S X_1 & X_2\ar@{->>}[l]^(.3){\sim}_(.3){\xi_2}\ar[r]^{d_1} & X_1}$ whose higher compositional
laws are encoded in the Segal maps $\xi_n$ for $n>2$. 

In a similar fashion, $v$-fibrant bisimplicial sets can carry a weak horizontal fractional structure as referred to in the Introduction. It is induced by acyclicity of their associated Bousfield maps defined as follows.
Consider the function $\gamma_i\colon[1]\rightarrow[n]$ mapping $0\mapsto 0$, $1\mapsto i$ and let
\[C_{n}:=\bigcup_{0<i}\gamma_i[\Delta^1]\]
be the $1$-skeletal cone whose pinnacle is the initial vertex $0\in\Delta^n$. We will refer to its 
edges as the \emph{initial edges} of $\Delta^n$ and let $c_{n}\colon C_{n}\hookrightarrow\Delta^n$ 
denote the canonical inclusion. Localizing $(s\mathbf{S},R_v)$ at the set of horizontally 
constant diagrams
\[B:=\{p_1^{\ast}(c_{n})\colon p_1^{\ast}(C_{n})\rightarrow p_1^{\ast}(\Delta^n)\mid n\geq 2\}\]
yields a model structure $(s\mathbf{S},B):=\mathcal{L}_B(s\mathbf{S},R_v)$. This model structure was introduced in 
\cite[6]{bergner2}, Bergner calls its fibrant objects \emph{Bousfield-Segal spaces}.
A $v$-fibrant bisimplicial set $X$ is $B$-local if and only if the fibrations
$c_{n}\setminus X\colon \Delta^n\setminus X\twoheadrightarrow C_{n}\setminus X$ are weak homotopy 
equivalences. Here, $C_{n}\setminus X\cong X_1\times_{X_0}\dots\times_{X_0}X_1$ is the $n$-fold fibre 
product of $X_1$ over $X_0$ along $d_1$ everywhere. We distinguish this pullback notationally by 
$X_1\times^B_{X_0}\dots\times^B_{X_0}X_1$. We define the \emph{Bousfield maps}
\[\beta_n\colon X_n\rightarrow X_1\times^B_{X_0}\dots\times^B_{X_0}X_1\]
of $X$ via $\beta_n:=c_{n}\setminus X$.

\begin{definition}\label{defbspace}
Let $X$ be a $v$-fibrant bisimplicial set. We say that $X$ is a \emph{Bousfield-Segal space} if the Bousfield maps
\begin{equation}\label{bmap}
\beta_n\colon X_n\rightarrow X_1\times^B_{X_0}\dots\times^B_{X_0}X_1
\end{equation}
are weak homotopy equivalences for all $n\geq 2$.
\end{definition}

Given a Bousfield-Segal space $X$, the fibration $\beta_2\colon X_2\twoheadrightarrow X_1\times_{X_0}^BX_1$ admits a 
section $\mu_2$ and thus yields a composite map
\begin{align}\label{equdeffrac}
\blank / \blank\colon  X_1\times_{X_0}^BX_1\xrightarrow{\mu_2}X_2\xrightarrow{d_0}X_1.
\end{align}
From now on we refer to this map as the \emph{fraction} operation associated to $X$.
On the horizontal simplicial sets $X_{\bullet m}$ it may be illustrated as follows.
\[
\vcenter{
\xymatrix{
                               & z & \\
 x\ar@/^/[ur]^{g}\ar@/_/[rr]_{f} & & y\\
  }}
  \mapsto
\vcenter{
  \xymatrix{
                               & z\ar@/^/[dr]^{f/g}\ar @{} [d] |{\mu_2(f,g)} & \\
 x\ar@/^/[ur]^{g}\ar@/_/[rr]_{f} & & y\\
  }}
  \mapsto
\vcenter{
  \xymatrix{
 z\ar@/^/[dr]^{f/g} & \\
 & y\\
 }}
  \]

\begin{definition}
Given a bisimplicial set $X$, for vertices $x\in X_{00}$ we write $1_x:=s_0 x$ and for $v,w\in X_{n0}$ we write
$v\sim w$ if $[v]=[w]\in\pi_{0}X_n$. If $X$ is $v$-fibrant and $x,y\in X_{00}$ are vertices, the \emph{hom-space} $X(x,y)$ 
denotes the pullback of $(d_1,d_0)\colon X_1\rightarrow X_0\times X_0$ along $(x,y)\in X_{00}\times X_{00}$. 
\end{definition}

The following Lemma shows that the operation (\ref{equdeffrac}) satisfies the axioms of an abstract fraction operation 
emerging in ordinary group theory in a many-sorted and homotopy-coherent manner.

\begin{lemma}\label{fractionlaws}
For any Bousfield-Segal space $X$ and $x,y,z\in X_{00}$, the fraction operation restricts to a map
\[\blank/\blank\colon X_1(x,y)\times X_1(x,z)\rightarrow X_1(z,y).\]
Then
\begin{enumerate}[label=(\arabic*)]
\item $f/f\sim 1_y$ for all vertices $f\colon x\rightarrow y$ in $X_{1}$,
\item $f/1_{x}\sim f$ for all vertices $f\colon x\rightarrow y$ in $X_{1}$,
\item $f/g\sim (f/h)/(g/h)$ for all triples $(f,g,h)\in X_1\times_{X_0}^B X_1\times_{X_0}^B X_1$.
\end{enumerate}
\end{lemma}
\begin{proof}
Straightforward calculation.
\end{proof}
The maps $\mu_2$ and $d_0$ are natural transformations of simplicial sets, hence the operation $\blank/\blank$ descends 
to homotopy classes. Therefore, for the family of sets
\[\mathrm{Ho}_B(X):=(\pi_0 X_1(x,y)\mid (x,y)\in X_{00}\times X_{00})\]
we obtain the following corollary.

\begin{corollary}\label{HoBSisgrpd}
The family $\mathrm{Ho}_B(X)$ together with the operation $\blank\circ\blank$, defined as the composite
\[\mathrm{Ho}_B(X)(y,z)\times\mathrm{Ho}_B(X)(x,y)\xrightarrow{}\mathrm{Ho}_B(X)(y,z)\times\mathrm{Ho}_B(X)(y,x)\xrightarrow{}\mathrm{Ho}_B(X)(x,z)\]
\[([g],[f])\mapsto[g]/([1_x]/[f])=[g/(1_x/f)],\]
is a groupoid.
\end{corollary}
\begin{proof}
Straightforward calculation.
\end{proof}

This fraction operation on Bousfield-Segal spaces is in its essence already present in \cite{bousfieldnotes}.
Corollary~\ref{HoBSisgrpd} says that the fraction operation on a Bousfield-Segal space $X$ induces an invertible 
composition on the quotient of $X$ under some of its homotopical data. In the course of the following two sections we 
lift this statement to the level of the homotopy-coherent data itself.

\section{Bousfield-Segal spaces are $B$-local Segal spaces}\label{secb=bs}

Despite the suggestive name it is not clear a priori that Bousfield-Segal spaces as defined 
in the previous section are in fact Segal spaces. In this section we show that Bousfield-Segal spaces are exactly
the $B$-local Segal spaces.

Let $X$ be a Bousfield-Segal space and recall the notation from (\ref{smap}) and (\ref{bmap}) 
for its associated Segal and Bousfield maps respectively. Then by Definition~\ref{defbspace}, its Bousfield maps 
$\beta_n\colon X_n\twoheadrightarrow X_1\times_{X_0}^B\dots\times_{X_0}^B X_1$ are acyclic 
fibrations. In order to show that $X$ is a Segal space, we have to derive that its 
Segal maps $\xi_n\colon X_n\twoheadrightarrow X_1\times_{X_0}^S\dots\times_{X_0}^S X_1$ are acyclic as well.

This means we want to show that the functor $\blank\setminus X\colon \mathbf{S}^{op}\rightarrow\mathbf{S}$ takes the 
spine inclusions $\mathrm{sp}_n$ to acyclic fibrations whenever it takes the left cone inclusions $c_n$ to such.
The proof can thus be boiled down to verifying that every class of morphisms in $\mathbf{S}$ which is subject to 
suitable closure properties contains the spine inclusions whenever it contains the left cone inclusions. Therefore, the 
discussion in this section will take place almost entirely in the category of simplicial sets.

We proceed in the following steps. Let $A$ be a class of arrows in $\mathbf{S}$ which contains the left cone inclusions 
and furthermore is saturated and satisfies 3-for-2 for monomorphisms (these are the suitable closure properties referred 
to above). Then, first, via a reduction to the left horn inclusions, we show that for every $n\geq 2$ the canonical 
embedding of $\Delta^n=N([n])$ into the free groupoid $I\Delta^n=N(I[n])$ generated by it is contained in $A$ as well. 
In fact, we will need only the case $n=2$ here, but the full result will be applied later in Section~\ref{secinvedges}. 
The case $n=2$ will be used to show that $c_2\in A\Leftrightarrow I[c_2]\in A$ and that
$\mathrm{sp}_2\in A\Leftrightarrow I[\mathrm{sp}_2]\in A$, where the maps on the right hand side are canonically induced 
inclusions into the free groupoid $I\Delta^2$. Now, while the automorphism group of the category $[2]$ is 
trivial, the groupoid $I[2]$ comes equipped with a non-trivial automorphism which interchanges its associated spine and 
cone inclusions $I[\mathrm{sp}_2]$ and $I[c_2]$. That means $I[c_2]\in A\Leftrightarrow I[\mathrm{sp}_2]\in A$, which 
together with the two equivalences above proves that $\mathrm{sp}_2\in A$. The case $n>2$ will then follow from $n=2$ by 
a factorization argument. 

With this road map in mind, we first relate the left cone inclusions and the left horn inclusions. The following 
lemma is a variation of \cite[Lemma 3.5]{jtqcatvsss} which is a similar statement for essential edges.

\begin{lemma}\label{ihicinterchange}
Let $A$ be a saturated class of morphisms in $\mathbf{S}$. Suppose further that $A$ has the right 
cancellation property for monomorphisms, i.e.\ $vu\in A$ and $u\in A$ imply $v\in A$ for all monomorphisms
$u,v\in\mathbf{S}$. Then $\{h_0^n\mid n\geq 2\}\subseteq A$ if and only if $\{c_{n}\mid n\geq 2\}\subseteq A$.
\end{lemma}

\begin{proof}
Let $k_n\colon C_{n}\rightarrow\Lambda^n_0$ be the canonical inclusion of simplicial sets, such that
the map $c_{n}\colon C_{n}\hookrightarrow\Delta^n$ factors through the inclusions
\[C_{n}\xrightarrow{k_n}\Lambda_0^n\xrightarrow{h_0^n}\Delta^n.\]
By assumption, it suffices to show that $k_n\in A$ for all $n\geq 2$ for both directions.
Therefore, we show that the inclusions $k_{n}$ can be constructed from the lower dimensional left 
cone inclusions $c_m$ by pasting them together recursively in such a way that whenever $c_m\in A$ for all $m<n$ we stay 
inside $A$ at each step along the way. The two implications will then be shown to follow from this construction in the 
last paragraph.

For $n\geq 2$ and $0<i\leq n$, consider the inclusion 
\begin{align}\label{equihicinterchange1}
C_{n}\hookrightarrow C_{n}\cup\bigcup_{0<j\leq i}d^j[\Delta^{n-1}].
\end{align}
Note that for $i=n$ this inclusion is exactly $k_n$. We show by induction that the inclusion (\ref{equihicinterchange1}) 
is contained in $A$ for all $0<i\leq n$ whenever the set $\{c_{m}\mid 2\leq m<n\}$ is contained in $A$.

For $n=2$ and $0<i\leq 2$ the inclusion (\ref{equihicinterchange1}) is exactly the identity
$k_2\colon C_2\rightarrow\Lambda^2_0$ and is hence contained in $A$.
Given $n\geq 2$, assume the inclusion (\ref{equihicinterchange1}) is contained in $A$ for every $0<i\leq n$. We 
now show that the inclusion
\[C_{n+1}\hookrightarrow C_{n+1}\cup\bigcup_{0<j\leq i}d^j[\Delta^{n}]\]
is contained in $A$ for every $0<i\leq n+1$. There is a pushout square
\begin{align}\label{diagihicinterchange1}
\begin{gathered}
\xymatrix{
C_{n}\ar[r]^(.3){\cong}_(.3){d^1}\ar@{^(->}[d]_{c_{n}} & C_{n+1}\cap d^1[\Delta^{n}]\ar@{^(->}[r]\ar@{^(->}[d]\ar @{}[dr]|(.7){\pos} & C_{n+1}\ar@{^(->}[d] \\
\Delta^{n}\ar[r]^(.4){\cong}_(.4){d^1} & d^1[\Delta^{n}]\ar@{^(->}[r] & C_{n+1}\cup d^1[\Delta^{n}]\\
}
\end{gathered}
\end{align}
where the boundaries $d^1$ in the left square are isomorphisms, because the coboundary
$d^1\colon[n]\rightarrow [n+1]$ is a monomorphism. This implies that the canonical inclusion
\[c_{(1,n+1)}\colon C_{n+1}\hookrightarrow C_{n+1}\cup d^1[\Delta^{n}]\]
is contained in $A$. Similarly, for $0<i\leq n$ we have a pushout square
\begin{align}\label{diagihicinterchange2}
\begin{gathered}
\xymatrix{
d^{i+1}[\Delta^{n}]\cap(C_{n+1}\cup\bigcup_{0<j\leq i}d^j[\Delta^{n}])\ar@{^(->}[r]\ar@{^(->}[d]\ar @{}[dr]|(.7){\pos} & C_{n+1}\cup\bigcup_{0<j\leq i}d^j[\Delta^{n}]\ar@{^(->}[d]^{} \\
d^{i+1}[\Delta^{n}]\ar@{^(->}[r] & C_{n+1}\cup\bigcup_{0<j\leq i+1}d^j[\Delta^{n}]
}
\end{gathered}
\end{align}
and isomorphisms
\begin{align}\label{diagihicinterchange3}
\xymatrix{
C_{n}\cup\bigcup_{0<j\leq i}d^j[\Delta^{n-1}]\ar[r]^(.4){\cong}_(.4){d^{i+1}}\ar@{^(->}[d]^{} & d^{i+1}[\Delta^{n}]\cap(C_{n+1}\cup\bigcup_{0<j\leq i}d^j[\Delta^{n}])\ar@{^(->}[d]\\
\Delta^{n}\ar[r]^(.4){\cong}_(.4){d^{i+1}} & d^{i+1}[\Delta^{n}].
}
\end{align}
Here, the upper boundary $d^{i+1}$ is an isomorphism, because
\begin{align*}
d^{i+1}[\Delta^{n}]\cap(C_{n+1}\cup\bigcup_{0<j\leq i}d^j[\Delta^{n}]) & = (d^{i+1}[\Delta^{n}]\cap C_{n+1})\cup\bigcup_{0<j\leq i}(d^{i+1}[\Delta^{n}]\cap d^j[\Delta^{n}]) \\
 & = (d^{i+1}[\Delta^{n}]\cap C_{n+1})\cup\bigcup_{0<j\leq i}d^{i+1}d^j[\Delta^{n-1}] \\
 & = d^{i+1}[C_{n}\cup\bigcup_{0<j\leq i}d^j[\Delta^{n-1}]].
\end{align*}
By the inductive hypothesis, the inclusion $C_{n}\hookrightarrow C_{n}\cup\bigcup_{0<j\leq i}d^j[\Delta^{n}]$ is
contained in $A$. But then, by the right cancellation property of $A$, the inclusion
$C_{n}\cup\bigcup_{0<j\leq i}d^j[\Delta^{n}]\hookrightarrow\Delta^n$ is contained in
$A$, too. Therefore, via (\ref{diagihicinterchange2}) and (\ref{diagihicinterchange3}), the canonical inclusion
\[c_{(i+1,n+1)}\colon C_{n+1}\cup\bigcup_{0<j\leq i}d^j[\Delta^{n}]\hookrightarrow C_{n+1}\cup\bigcup_{0<j\leq i+1}d^j[\Delta^{n}]\]
is contained in $A$ for every $0<i\leq n+1$. It follows that the composition
\[c_{(i+1,n+1)}\circ\dots c_{(2,n+1)}\circ c_{(1,n+1)}\colon C_{n+1}\hookrightarrow\bigcup_{0<j\leq i+1}d^j[\Delta^{n}]\]
is contained in $A$. In particular, since $k_{n+1}$ is the composition of all $c_{(i,n+1)}$ for $0<i\leq n+1$, the map 
$k_{n+1}$ is contained in $A$. This finishes the induction.

Thus, on the one hand, $\{c_n\mid n\geq 2\}\subseteq A$ implies $\{k_n\mid n\geq 3\}\subseteq A$, and
$k_2=\mathrm{id}_{C_2}$ is contained in $A$ trivially. On the other hand, in order to prove that 
$\{h_0^n\mid n\geq 2\}\subseteq A$ implies $\{k_n\mid n\geq 2\}\subseteq A$ and hence $\{c_n\mid n\geq 2\}\subseteq A$, 
assume that $A$ contains all left horn inclusions. Since $C_{2}=\Lambda_0^2$ 
and $h_0^2=c_{2}$, the inclusion $c_{2}$ is contained in $A$.
Suppose $n\geq 2$ and $c_{m}\in A$ for all $2\leq m\leq n$. As we have seen above, this implies $k_{n+1}\in A$. This in 
turn implies $c_{n+1}\in A$, because $c_{n+1}=h_0^{n+1}\circ k_{n+1}$.	
\end{proof}

\begin{corollary}\label{corihicinterchange}
Let $X\in s\mathbf{S}$ be $v$-fibrant. The following two conditions are equivalent.
\begin{enumerate}[label=(\arabic*)]
\item $c_{n}\setminus X$ is an acyclic fibration in $\mathbf{S}$ for all $n\geq 2$.
\item $h_0^n\setminus X$ is an acyclic fibration in $\mathbf{S}$ for all $n\geq 2$.
\end{enumerate}
\end{corollary}

\begin{proof}
Let $X$ be $v$-fibrant. The class
\[A:=\{f\in\mathbf{S}\mid f\text{ is a monomorphism and }f\setminus X\text{ is an acyclic fibration}\}\]
has the right cancellation property for monomorphisms. It is saturated by Proposition \ref{boxdividadj} and 
the fact that the class of monomorphisms in $\mathbf{S}$ is saturated. Therefore, the two statements are equivalent by 
Lemma \ref{ihicinterchange}.
\end{proof}

Next, we show that the canonical embeddings $\Delta^n\rightarrow I\Delta^n$ into the nerve $I\Delta^n:=N(I[n])$ of the 
free groupoid generated by $[n]$ are cellular cofibrations with respect to 
the set $\{h_0^n\mid n\geq 2\}$. By that we mean that each map $\Delta^n\rightarrow I\Delta^n$ 
is the transfinite composition of pushouts of coproducts of such horn inclusions. In the case $n=1$, this was noted 
by Rezk in \cite[11]{rezk} using a combinatorial description of the simplices in $I\Delta^1$.
Although the combinatorics of the $I\Delta^n$ for $n\geq 2$ is considerably more complicated, we 
can make use of a suitable alternative description of the nerve of a groupoid, given for example in 
\cite{moerdijkgroupcompletion}.
Therefore, consider diagrams in $\mathbf{S}(C_m,I\Delta^n)$ depicted as follows.
\begin{align}\label{diagconicalnerve}
\begin{gathered}
\xymatrix{
& & c_0\ar@/_/[dll]_{h_1}\ar@/_/[dl]^{h_2}\ar@/^/[dr]^{h_m} &  \\ 
c_1 & c_2 & \dots & c_m
}
\end{gathered}
\end{align}
Every such conical diagram is uniquely determined by its sequence $(c_i\mid 0\leq i\leq m)$ of objects in $[n]$.
Indeed, we obtain a bijection
\[(I\Delta^n)_m\cong\mathbf{S}(C_m,I\Delta^n)\]
which sends an $m$-simplex $c_0\xrightarrow{f_1}c_1\xrightarrow{f_2}\dots\xrightarrow{f_m}c_m$ to the diagram
$(c_0\xrightarrow{f_i\dots f_1}c_i)_{i\leq m}$. Under this identification, the boundaries and degeneracies of an
$m$-simplex in $I\Delta^n$ represented as a cone $(c_0\xrightarrow{h_i} c_i)_{i\leq m}\in\mathbf{S}(C_m,I\Delta^n)$ are 
given as follows.
\begin{align*}
d_j((c_0\xrightarrow{h_i} c_i)_{i\leq m}) & =
\begin{cases}
   (c_0\xrightarrow{h_i} c_i)_{i\leq m,i\not=j}        & \text{if } j > 0 \\
   (c_1\xrightarrow{h_i\dots h_2} c_i)_{1< i\leq m}        & \text{if } j = 0
  \end{cases}\\
s_j((c_0\xrightarrow{h_i} c_i)_{i\leq m}) & =
\begin{cases}
   (h_1,\dots,h_j,h_j,h_{j+1},\dots,h_m)        & \text{if } j > 0 \\
   (1_{c_0},h_1,\dots,h_m)        & \text{if } j = 0
  \end{cases}
\end{align*}

\begin{proposition}\label{freegrpdinv}
The canonical inclusions $\{\Delta^n\rightarrow I\Delta^n\mid n\geq 1\}$ are cellular cofibrations with respect to the 
set $\{h_0^n\mid n\geq 2\}$.
\end{proposition}

\begin{proof}
Let $n\geq 1$. The subobject $\Delta^n\subset I\Delta^n$ consists exactly of those conical diagrams 
(\ref{diagconicalnerve}) such that the sequence $(c_i\mid 0\leq i\leq m)$ is monotone in the linear order $[n]$. We 
will add the non-degenerate simplices in its complement step by step via recursion. Therefore,
we define a filtration $\{F^{(m)}\mid m\geq 1\}$ of $I\Delta^n$ with $F^{(1)}:=\Delta^n$ such that the 
inclusions $F^{(m)}\rightarrow F^{(m+1)}$ are cellular for the left horn inclusions.

Note that every 1-simplex in $I\Delta^n$ of the form $0 \rightarrow c_1$ for $c_1\in[n]$ is contained in $F^{(1)}$. 
Thus, we may assume that $F^{(m)}\subseteq I\Delta^{n}$ has been defined such that every diagram of the form 
$(0\rightarrow c_i\mid 1\leq i\leq m)$ is contained in $F^{(m)}$.

Then every non-degenerate non-monotone sequence $(c_i\mid 0\leq i\leq m+1)$ in $[n]$ with $c_0=0$ 
determines a unique map $\Lambda^{m+1}_0\rightarrow I\Delta^n$. This map factors through the subobject
$F^{(m)}$ since, for every $0<i\leq m+1$, the boundary
\[d_i((c_i\mid 0\leq i\leq m+1))=(0,\dots,\hat{c_i},\dots, c_{m+1})\]
is contained in $F^{(m)}$ by assumption.

Let $T_{m+1}$ be the set of all non-degenerate
non-monotone sequences $(c_i\mid 0\leq i\leq m+1)$ with $c_0=0$, and let $F^{(m+1)}$ be the pushout
\[\xymatrix{
\coprod_{c\in T_{m+1}}\Lambda^{m+1}_0\ar@{^(->}[r]^(.55){\coprod h^{m+1}_0}\ar[d]_{c}\ar@{}[dr]|(.7){\pos} & \coprod_{c\in T_{m+1}}\Delta ^{m+1}\ar[d] \\
F^{(m)}\ar@{^(->}[r] & F^{(m+1)}.
}\]
We obtain a natural map $F^{(m+1)}\rightarrow I\Delta^n$ via the inclusion $F^{(m)}\subseteq I\Delta^n$ on the one hand, 
and the unique maps $c\colon\Delta^{m+1}\rightarrow I\Delta^n$ extending 
$c\colon\Lambda^{m+1}_0\rightarrow I\Delta^n$ for $c\in T_{m+1}$ on the other hand.
This map $F^{(m+1)}\rightarrow I\Delta^n$ is a monomorphism as can be verified by case distinction. The 
subobject $F^{(m+1)}\subseteq I\Delta^n$ contains every sequence of the form $(0\rightarrow c_i\mid 1\leq i\leq m+1)$  
in $[n]$ by construction.

Every inclusion $F^{(m)}\rightarrow F^{(m+1)}$ is cellular for the set
$\{h_0^n\mid n\geq 2\}$, and so is the composite inclusion $\Delta^1\rightarrow\bigcup_{1\leq m}F^{(m)}$. We are 
left to show that every simplex in $I\Delta^n$ is contained in some $F^{(m)}$.

Therefore, let $(c_i\mid 0\leq i\leq m)$ be a non-degenerate $m$-simplex in $I\Delta^n$. If the sequence is monotone 
or $c_0=0$, it is contained in $F^{(m)}$. Otherwise, $c_0>0$ and the sequence $(0,c_0,\dots,c_m)$ is 
contained in the set $T_{m+1}$. It follows that $(0,c_0,\dots,c_m)$ is an $(m+1)$-simplex in $F^{(m+1)}$, 
and so $d_0((0,c_0,\dots,c_m))=(c_0,\dots,c_m)$ is contained in $F^{(m+1)}$ as well. This finishes the proof.
\end{proof}

\begin{corollary}\label{conetospine}
Let $A$ be a saturated class of morphisms in $\mathbf{S}$ with 3-for-2 for monomorphisms (i.e.\ with both the left and 
right cancellation property for monomorphisms). Then $B=\{c_n\mid n\geq 2\}\subseteq A$ implies
$S=\{\mathrm{sp}_n\mid n\geq 2\}\subseteq A$.
\end{corollary}

\begin{proof}
We show the cases $n=2$ and $n>2$ separately.

For $n=2$, we denote the non-degenerate 2-cell in $\Delta^2$ by $[0]\xrightarrow{f}[1]\xrightarrow{g}[2]$. Let
$I[\mathrm{Sp}_2]\subset I\Delta^2$ and $I[C_2]\subset I\Delta^2$ be the respective pushouts
$I\Delta^1\cup_{\Delta^0} I\Delta^1\subset I\Delta^2$ along the obvious pairs of boundaries.

Then the 
assignment $0\mapsto 2$, $1\mapsto 0$, $2\mapsto 1$, $f\mapsto (gf)^{-1}$, $g\mapsto f$ generates an automorphism
$a\colon I\Delta^2\rightarrow I\Delta^2$ with $a[I[\mathrm{Sp}_2]]=I[C_2]$. We obtain a diagram of inclusions as follows.
\[\xymatrix{
\mathrm{Sp}_2\ar[d]^{\mathrm{sp}_2}\ar[r] & I[\mathrm{Sp}_2]\ar[r]^{a}_{\cong}\ar[d]^{I[\mathrm{sp}_2]} & I[C_2]\ar[d]^{I[c_2]} & C_2\ar[d]^{c_2}\ar[l] \\
\Delta^2\ar[r] & I\Delta^2\ar[r]^{a}_{\cong} & I\Delta^2 & \Delta^2\ar[l]
}\]
The bottom inclusions are contained in $A$ by Lemma~\ref{ihicinterchange} and Proposition~\ref{freegrpdinv}. The top 
inclusions are pushouts of the inclusion $\Delta^1\rightarrow I\Delta^1$ and it follows via saturation that they are 
contained in $A$ as well. The map $c_2$ is contained in $A$ by assumption, and hence it 
follows by left and right cancellation that $\mathrm{sp}_2\in A$.

For $n>2$, let $\Delta^{\{i,j,k\}}\subset\Delta^n$ denote the subobject generated by the 2-cell
$[i]\rightarrow [j]\rightarrow [k]$ whenever $i\leq j\leq k$. Let $(\mathrm{SpC})_n\subset \Delta^n$ be the 
subobject given by the union
\[(\mathrm{SpC})_n:=\bigcup_{0<i<n}\Delta^{\{0,i,i+1\}}.\]
It contains all essential and all initial edges of $\Delta^n$, so both the spine and left cone inclusions 
factor through $(\mathrm{SpC})_n$. 
\begin{align}\label{CZfac}
\begin{gathered}
\xymatrix{
\mathrm{Sp}_n\ar@/_/[ddr]_{\mathrm{sp}_n}\ar[dr] & & C_n\ar@/^/[ddl]^{c_n}\ar[dl] \\
 & (\mathrm{SpC})_n\ar[d] & \\
 & \Delta^n & 
}
\end{gathered}
\end{align}
The inclusion $\mathrm{Sp}_n\rightarrow(\mathrm{SpC})_n$ is a finite composition of pushouts of
$\mathrm{sp}_2\colon \mathrm{Sp}_2\rightarrow\Delta^2$, and, likewise, the inclusion $C_n\rightarrow(\mathrm{SpC})_n$ is 
a finite composition of pushouts of $c_2\colon C_2\rightarrow\Delta^2$. It follows that both these inclusions are 
contained in $A$, and so does $c_n\colon C_n\rightarrow\Delta^n$ by assumption. It follows that
$(\mathrm{SpC})_n\rightarrow\Delta^n$ lies in $A$ by right cancellation, and hence so does the composition
$\mathrm{Sp}_n\rightarrow(\mathrm{SpC})_n\rightarrow\Delta^n$.
\end{proof}

\begin{theorem}\label{B=BS}
Every Bousfield-Segal space is a Segal space. In particular, the model structures $(s\mathbf{S},B)$ and
$\mathcal{L}_B(s\mathbf{S},S)$ coincide.
\end{theorem}

\begin{proof}
Let $X$ be a Bousfield-Segal space and consider the class
\[A:=\{f\in\mathbf{S}\mid f\text{ is a monomorphism and }f\setminus X\text{ is an acyclic fibration}\}.\]
The class $A$ is saturated, because the class of acyclic fibrations in $(\mathbf{S},\mathrm{Kan})$ is closed under 
pullbacks, retracts and sequential limits. It satisfies 3-for-2 for monomorphisms, because $\blank\setminus X$ maps all 
monomorphisms to fibrations and the class of weak equivalences in $(\mathbf{S},\mathrm{Kan})$ satisfies 3-for-2. We have 
$B\subseteq A$ by assumption and hence $S\subseteq A$ by Corollary~\ref{conetospine}. Therefore, $X$ is a Segal space.

In other words, every fibrant object in $(s\mathbf{S},B)$ is also fibrant in $\mathcal{L}_B(s\mathbf{S},S)$. 
This implies that the left Bousfield localizations $(s\mathbf{S},B)$ and $\mathcal{L}_B(s\mathbf{S},S)$ have the same 
class of fibrant objects, and hence coincide.
\end{proof}

\begin{remark}
One also can show Theorem~\ref{B=BS} in a more algebraic fashion. The algebraic laws of the fraction 
operation described in Lemma~\ref{fractionlaws} hold in a homotopy coherent manner and can be used to construct a 
homotopy inverse to the Segal map $\xi_2$ of any given Bousfield-Segal space. Acyclicity of the Segal maps $\xi_n$ for 
$n>2$ can then be proven as in Corollary~\ref{conetospine}, or by induction similar to Lemma~\ref{ihicinterchange}.
\end{remark}

Theorem~\ref{B=BS} implies that the constructions for Segal spaces from \cite{rezk} apply to the class of
Bousfield-Segal spaces. For example, every Bousfield-Segal space $X$ comes equipped with a homotopy category
$\mathrm{Ho}(X)$ as constructed in \cite[5.5]{rezk}. This induces a notion of Dwyer-Kan equivalences (\cite[7.4]{rezk}) 
between Bousfield-Segal spaces $X$ and $Y$, which are exactly those maps between $X$ and $Y$ to be inverted in the model 
structure for complete Segal spaces (\cite[Theorem 7.7]{rezk}).

Recall the groupoid $\mathrm{Ho}_B(X)$ associated to $X$ in Corollary~\ref{HoBSisgrpd}.

\begin{corollary}\label{B=BSspaces}
For any Bousfield-Segal space $X$, the categories $\mathrm{Ho}(X)$ and $\mathrm{Ho}_B(X)$ coincide. In particular,
$\mathrm{Ho}(X)$ is a groupoid.
\end{corollary}
\begin{proof}
Let $X$ be a Bousfield-Segal space. The families $\mathrm{Ho}_B(X)$ and $\mathrm{Ho}(X)$ of sets coincide and 
have the same identity, so we have to show that the two corresponding compositions induced by the Bousfield and Segal 
maps coincide, too. Let $\eta_2$ be a section to
$\xi_2\colon X_2\overset{\sim}{\twoheadrightarrow} X_1\times_{X_0}^S X_1$, such that
$\blank\circ_S\blank:=d_1\eta_2$ is a composition for the Segal space $X$. For any two morphisms $f\in X(x,y)$ 
and $g\in X(y,z)$, the associated inner 3-horn
$\eta_2(f,g)\cup\mu_2(1_x,f)\cup\mu_2(g,1_x/f)\colon\Lambda_1^3\rightarrow X_{\bullet 0}$ of the form
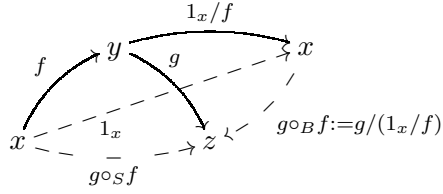
\[\xymatrix{
 & y\ar@/^/[dr]^(.4){g}\ar@/^/[rr]^{1_x/f} & & x\ar@{-->}@/^/[dl]^{g\circ_B f:=g/(1_x/f)} \\
 x\ar@/^/[ur]^{f}\ar@{-->}@/_/[rr]_{g\circ_Sf}\ar@{-->}[urrr]_(0.3){1_x} & & z & \\
}\]
has a lift $L(f,g)\colon\Delta^3\rightarrow X_{\bullet 0}$. Both the simplex
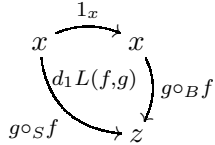
\[\xymatrix{
x\ar@/^/[r]^{1_x}\ar@/_1pc/[dr]_{g\circ_S f}\ar@{}[dr]|(.4){\hspace{1pc}d_1L(f,g)} & x\ar@/^/[d]^{g\circ_B f}\\
 & z
}\]
and $s_0(g\circ_S f)$ lie in the contractible fiber $\beta_2^{-1}(g\circ_S f,1_x)$. In particular, $d_1L(f,g)$ and 
$s_0(g\circ_S f)$ lie in the same connected component of $X_2$. By naturality of $d_0$, we obtain
\[[g\circ_B f]=[d_0d_1L(f,g)]=[d_0s_0(g\circ_S f)]=[g\circ_S f]\]
in $\pi_0X_1(x,z)=\mathrm{Ho}(X)(x,z)=\mathrm{Ho}_B(X)(x,z)$.
\end{proof}

\section{Further characterizations}\label{secinvedges}

In this section we prove that $(s\mathbf{S},B)$ can be obtained as the left Bousfield localization of $(s\mathbf{S},S)$ 
at one single map which exhibits Bousfield-Segal spaces as exactly those Segal spaces whose homotopy category is a 
groupoid. In algebraic terms, this means that every homotopy-coherent fraction operation on a space $X_0$ induces a 
unique invertible homotopy-coherent composition operation on $X_0$ and vice versa.
Furthermore, we construct a right adjoint to the inclusion of Bousfield-Segal spaces into Segal spaces, and 
thereby define the \emph{core} of a Segal space.

\begin{proposition}\label{BSChar3}
A Segal space $X$ is a Bousfield-Segal space if and only if the Bousfield map
\[\beta_2=c_{2}\setminus X\colon X_2\rightarrow X_1\times_{X_0}^BX_1\]
is an acyclic fibration. In particular, the model structures $\mathcal{L}_{c_{2}}(s\mathbf{S},S)$ and
$(s\mathbf{S},B)$ coincide. 
\end{proposition}
\begin{proof}
Since both model structures are left Bousfield localizations of the same Reedy structure, we only have to 
compare their fibrant objects. Every Bousfield-Segal space is $c_2$-local by definition and fibrant in
$(s\mathbf{S},S)$ by Theorem~\ref{B=BS}. Thus, one direction is immediate. Vice versa, we have to show that fibrant 
objects in $\mathcal{L}_{c_{2}}(s\mathbf{S},S)$ are $p_1^{\ast}c_{n}$-local for all $n\geq 2$. Therefore, 
consider the class
\[A:=\{f\in\mathbf{S}\mid f\text{ is a monomorphism and }p_1^{\ast}f\text{ is an acyclic cofibration in }\mathcal{L}_{c_{2}}(s\mathbf{S},S)\}.\]
Since the fibrant objects in $\mathcal{L}_{c_{2}}(s\mathbf{S},S)$ are $f$-local for all $f\in A$, we want to show 
that $(c_n\mid n\geq 2)\subseteq A$.
The class $A$ is saturated and has the right cancellation property for monomorphisms. By construction
$S\cup\{c_{2}\}$ is a subset of $A$. Via the common factorization of $c_n$ and $\mathrm{sp}_n$ for $n\geq 2$ given in 
Diagram (\ref{CZfac}), it follows that the left cone inclusions $c_n$ for $n>2$ are contained in $A$ as well (in the 
same way the proof of Corollary~\ref{conetospine} showed that $B\cup\{\mathrm{sp}_2\}\subset A$ implies $\mathrm{sp}_n\in A$ 
for all $n\geq 2$).
\end{proof}

\begin{remark}
One can define both Segal objects and Bousfield-Segal objects in every model category $\mathbb{M}$, via imposing 
acyclicity of associated Segal and Bousfield maps. The functor
$\blank\setminus X\colon\mathbf{S}^{op}\rightarrow\mathbb{M}$ can be defined accordingly for every 
Reedy fibrant simplicial object $X\in\mathbb{M}^{\Delta^{op}}$. Then the proofs of Theorem~\ref{B=BS} and
Proposition~\ref{BSChar3} in fact show that Bousfield-Segal objects in $\mathbb{M}$ are exactly the Segal objects with 
invertible Bousfield map $\beta_2$. 
\end{remark}

Let $X_{\operatorname{hoequiv}}\subseteq X_1$ denote the full simplicial subset of \emph{homotopy equivalences} in
$X$, generated by those $f\in (X_{1})_0$ which become isomorphisms in $\mathrm{Ho}(X)$.

\begin{corollary}\label{Charvia2}
A Segal space $X$ is a Bousfield-Segal space if and only if its associated homotopy category $\mathrm{Ho}(X)$ is a 
groupoid.
\end{corollary}
\begin{proof}
One direction was shown in Corollary~\ref{B=BSspaces}. For the other direction, let $X$ be a Segal space and assume
$\mathrm{Ho}(X)$ is a groupoid. That means that $X_{\mathrm{hoequiv}}=X_1$, and so it follows from
\cite[Lemma 11.6]{rezk} that the Bousfield map $\beta_2\colon X_2\twoheadrightarrow X_1\times_{X_0}X_1$ is a weak 
equivalence. Hence $X$ is a Bousfield-Segal space by Proposition~\ref{BSChar3}.
\end{proof}

Whenever $X$ is a Segal space, the fibration $I\Delta^1\setminus X\twoheadrightarrow X_1$ induced by the canonical 
inclusion $e\colon\Delta^1\rightarrow I\Delta^1$ factors through the subspace $X_{\mathrm{hoequiv}}$, and Rezk showed in 
\cite[Theorem 6.2]{rezk} that the resulting fibration
$e\setminus X\colon I\Delta^1\setminus X\rightarrow X_{\mathrm{hoequiv}}$ is acyclic.

\begin{corollary}\label{rezktechnical}
A Segal space $X$ is a Bousfield-Segal space if and only if the fibration
$I\Delta^1\setminus X\twoheadrightarrow X_1$ is acyclic. In particular, the model structures
$(s\mathbf{S},B)$ and $\mathcal{L}_{p_1^{\ast}(e)}(s\mathbf{S},S)$ coincide.
\end{corollary}
\begin{proof}
If $X$ is a Bousfield-Segal space, the fibration $I\Delta^1\setminus X\twoheadrightarrow X_1$ is acyclic by
Proposition~\ref{freegrpdinv}. Vice versa, if $X$ is a Segal space and the fibration
$I\Delta^1\setminus X\twoheadrightarrow X_1$ is acyclic, it is surjective as well. Since it factors through the 
subobject $X_{\mathrm{hoequiv}}$, this implies $X_{\mathrm{hoequiv}} = X_1$. It follows that $\mathrm{Ho}(X)$ is a 
groupoid and hence $X$ is a Bousfield-Segal space by Corollary~\ref{B=BS}.
\end{proof}

\begin{remark}
Corollary~\ref{rezktechnical} justifies the motivating analogies stated in the Introduction, which related the 
construction of Bousfield-Segal spaces from Segal spaces to both the construction of $\mathbf{Gpd}$ from
$\mathbf{Cat}$ and the construction of $(\mathbf{S},\mathrm{Kan})$ from $(\mathbf{S},\mathrm{QCat})$. In a similar vein,
Corollary~\ref{Charvia2} yields an analogy to Joyal's criterion for a quasi-category to be Kan complex via its 
homotopy category (for instance as presented in \cite[Proposition 1.2.4.3 and 1.2.5.1]{luriehtt}).

It follows that a Segal space $X$ is a Bousfield-Segal space if and only if its rows are Kan complexes. That means
Bousfield-Segal spaces are exactly the Segal spaces horizontally fibrant in the projective model structure over
$(\mathbf{S},\mathrm{Kan})$.
\end{remark}

At this point we have seen in various ways that the inclusion $(s\mathbf{S},B)\rightarrow(s\mathbf{S},S)$ is part of a 
left Bousfield localization, and as such it is the right adjoint in a homotopy localization in the sense of
\cite[Definition 7.16]{jtqcatvsss}. 
We conclude this section with the construction of a further right adjoint to this inclusion, given by a core 
construction that associates to a Segal space $X$ the largest Bousfield-Segal space $\mathrm{Core}(X)$ contained in $X$.
We proceed in the following two steps. 

First, we construct a functor ``Core'' as the strict right adjoint to the inclusion
\begin{align}\label{equbsincl}
\iota\colon\text{Bousfield-Segal}\hookrightarrow\text{Segal}
\end{align}
of the full subcategory of Bousfield-Segal spaces into the full subcategory of Segal spaces in $s\mathbf{S}$. We do so 
by an elementary procedure which takes the subcollection $X_{\mathrm{hoequiv}}\subseteq X_1$ of equivalences in a 
Segal space $X$ and assigns it the largest sub-bisimplicial set $\mathrm{Core}(X)\subseteq X$ generated from it.

Second, we show that the colocalization $(\iota,\mathrm{Core})$ can be extended to a homotopy colocalization of the form
\[(L,R)\colon(s\mathbf{S},\mathrm{B})\rightarrow (s\mathbf{S},\mathrm{S}).\]
This means $(L,R)$ is a Quillen pair such that $L(X)\simeq \iota(X)$ in $(s\mathbf{S},\mathrm{S})$ whenever $X$ is a 
Bousfield-Segal space, $R(X)\simeq\mathrm{Core}(X)$ in $(s\mathbf{S},\mathrm{B})$ whenever $X$ is a Segal space, and 
such that the derived unit of the adjunction is a natural weak equivalence in $(s\mathbf{S},\mathrm{B})$.
The first step is given by the following lemma.

\begin{lemma}\label{lemmacorestrict}
The inclusion (\ref{equbsincl}) has a right adjoint.
\end{lemma}
\begin{proof}
Let $X$ be a Segal space and let $X_{\mathrm{hoequiv}}\subseteq X_1$ be the space of 
equivalences in $X$. We can define a bisimplicial set $\mathrm{Core}(X)$ recursively as follows. Let
$\mathrm{Core}(X)_0=X_0$, let
\[\mathrm{Core}(X)_1=X_{\mathrm{hoequiv}}\subseteq X_1,\]
and let the corresponding two boundaries and the degeneracy between these two objects be induced directly from those of 
$X$ by restriction.

For $n\geq 2$ define $\mathrm{Core}(X)_n\subseteq X_n$ to be the subspace generated by the vertices of $X_{n0}$ 
whose horizontal 1-boundaries lie in $X_{\mathrm{hoequiv}}$.
The boundaries and degeneracies of $\mathrm{Core}(X)$ are directly obtained by restriction of the respective boundaries 
and degeneracies of $X$, so we obtain a bisimplicial subset $\mathrm{Core}(X)\subseteq X$.
By construction, we obtain cartesian squares between the matching maps for all $n\geq 2$ as follows.
\begin{align}\label{diaglemmacorestrict1}
\begin{gathered}
\xymatrix{
\mathrm{Core}(X)_n\ar@{->>}[d]^{\partial}\ar@{^(->}[r]\ar@{}[dr]|(.3){\pbs} & X_n\ar@{->>}[d]^{\partial} \\
\partial\Delta^n\setminus\mathrm{Core}(X)\ar@{^(->}[r]  & \partial\Delta^n\setminus X
}
\end{gathered}
\end{align}
Since $X_{\mathrm{hoequiv}}\subseteq X_1$ is closed under composition, we furthermore obtain cartesian squares between 
the respective Segal maps as well.
\begin{align}\label{diaglemmacorestrict2}
\begin{gathered}
\xymatrix{
\mathrm{Core}(X)_n\ar@{->>}[d]^{\xi_n}\ar@{^(->}[r]\ar@{}[dr]|(.3){\pbs} & X_n\ar@{->>}[d]^{\xi_n} \\
\mathrm{Core}(X)_1\times_{X_0}^S\dots\times_{X_0}^S\mathrm{Core}(X)_1\ar@{^(->}[r]  & X_1\times_{X_0}^S\dots\times_{X_0} ^S X_1
}
\end{gathered}
\end{align} 
It follows that $\mathrm{Core}(X)$ is Reedy fibrant by (\ref{diaglemmacorestrict1}) and satisfies the Segal 
conditions by (\ref{diaglemmacorestrict2}). Furthermore, by construction we have
\[\mathrm{Core}(X)_{\mathrm{hoequiv}}=X_{\mathrm{hoequiv}}=\mathrm{Core}(X)_1,\]
so $\mathrm{Core}(X)$ is a Bousfield-Segal space by Corollary~\ref{Charvia2}. It is straight forward to show that this 
construction extends to a functor
\[\mathrm{Core}\colon\mathrm{Segal}\rightarrow\mathrm{Bousfield\text{-}Segal}.\]
which is right adjoint to the inclusion $\mathrm{Bousfield\text{-}Segal}\subset\mathrm{Segal}$.
\end{proof}

Towards the second step, note that for any Segal space $X$ the first row $\mathrm{Core}(X)_{\bullet 0}$ is the largest 
Kan complex $J(X_{\bullet 0})$ contained in the quasi-category $X_{\bullet 0}$ as constructed in
\cite[Proposition 1.16]{jtqcatvsss}). So the given right adjoint to the inclusion (\ref{equbsincl}) is a direct 
generalization of the right adjoint to the inclusion of Kan complexes into quasi-categories. Joyal also shows in \cite{joyalqcatslecture} that the functor
$J\colon\mathrm{QCat}\rightarrow\mathrm{Kan}$ can be extended to the Quillen right adjoint $k^!$ of a homotopy 
colocalization $(k_!,k^!)\colon (\mathbf{S},\mathrm{QCat})\rightarrow (\mathbf{S},\mathrm{Kan})$. Dual to the notion of 
homotopy localization, a Quillen pair is a homotopy colocalization if the left derived of its left adjoint is fully 
faithful on the underlying homotopy categories. 

Analogously, one can extend the core construction of a Segal space to a homotopy colocalization between the respective 
model structures.  
Therefore, consider the functor $k\bBox\mathrm{id}\colon\Delta\times\Delta\rightarrow s\mathbf{S}$ for
$k([n])=I\Delta^n$, together with its left Kan extension
$(k\bBox\mathrm{id})_!\colon s\mathbf{S}\rightarrow s\mathbf{S}$. It has a right adjoint $(k\bBox\mathrm{id})^!$, given 
by
\[(k\bBox\mathrm{id})^!(X)_{nm}=s\mathbf{S}(I\Delta^n\bBox \Delta^m,X).\]

The right adjoint $(k\bBox\mathrm{id})^!$ is a simplicially enriched functor with respect to the hom-spaces
$\mathrm{Hom}_2$ given in Proposition~\ref{prophom2}, and preserves the associated tensors (described in 
\cite[Proposition 2.4]{jtqcatvsss}) as well. Hence the pair $((k\bBox\mathrm{id})_!,(k\bBox\mathrm{id})^!)$ gives rise 
to a simplicially enriched adjoint pair on $s\mathbf{S}$.

\begin{proposition}\label{propBScoloc}
The functor $k\bBox\mathrm{id}\colon\Delta\times\Delta\rightarrow s\mathbf{S}$ induces a Quillen pair
\begin{align}\label{BScoloc}
\xymatrix{
(k\bBox\mathrm{id})_!\colon(s\mathbf{S},B)\ar@/^/[r]\ar@{}[r]|{\bot} & (s\mathbf{S},S)\colon (k\bBox\mathrm{id})^!.\ar@/^/[l]
}
\end{align}
It comes together with a natural transformation $\beta\colon(k\bBox\mathrm{id})^!\rightarrow\mathrm{id}$ which factors 
through a vertical Reedy weak equivalence $\beta_X\colon (k\bBox\mathrm{id})^!(X)\twoheadrightarrow\mathrm{Core}(X)$ 
whenever $X$ is a Segal space.

\end{proposition}
\begin{proof}
The inclusions $[n]\rightarrow I[n]$ induce a natural transformation
$\beta\colon(k\bBox\mathrm{id})^!\rightarrow\mathrm{id}$ by precomposition. It factors through $\mathrm{Core}(X)$ 
whenever $X$ is a Segal space, because $(\beta_X)_1$ is exactly the fibration 
$e\setminus X\colon I\Delta^1\setminus X\twoheadrightarrow X_1$. Generally, the maps
\[(\beta_X)_n\colon(k\bBox\mathrm{id})^!(X)_n\rightarrow \mathrm{Core}(X)_n\]
are isomorphic to the fibrations $I\Delta^n\setminus \mathrm{Core}(X)\rightarrow \Delta^n\setminus\mathrm{Core}(X)$ and 
hence are acyclic by Proposition~\ref{freegrpdinv}. Thus the natural transformation $\beta$ yields a levelwise acyclic 
fibration $\beta_X\colon(k\bBox\mathrm{id})^!(X)\rightarrow \mathrm{Core}(X)$ whenever $X$ is a Segal space. This 
shows that the right adjoint $(k\bBox\mathrm{id})^!$ is an extension of the core construction.

The left adjoint $(k\bBox\mathrm{id})_!$ furthermore preserves Reedy cofibrations and Reedy acyclic 
cofibrations as can be verified on the generating sets $\mathcal{I}_v$ and $\mathcal{J}_v$. Thus, the two functors yield 
a Quillen pair from the vertical Reedy model structure to itself. In order for them to descend to a Quillen pair between 
the two localizations as in (\ref{BScoloc}) we have to show that the right adjoint $(k\bBox\mathrm{id})^!$ maps 
every Segal space $X$ to a Bousfield-Segal space.
But we have seen that for every Segal space $X$, the map
$\beta_X\colon (k\bBox\mathrm{id})^!(X)\rightarrow\mathrm{Core}(X)$ is a Reedy weak equivalence. We know that
$\mathrm{Core}(X)$ is a Bousfield-Segal space and that the domain $(k\bBox\mathrm{id})^!(X)$ is Reedy fibrant,
so $(k\bBox\mathrm{id})^!(X)$ is a Bousfield-Segal space as well.
\end{proof}

\begin{theorem}
The pair (\ref{BScoloc}) is a homotopy colocalization. It furthermore comes together with a natural transformation
$\alpha\colon\mathrm{id}\rightarrow(k\bBox\mathrm{id})_!$ such that $\alpha_X$ is a weak equivalence in
$(s\mathbf{S},\mathrm{S})$ whenever $X$ is a Bousfield-Segal space.
\end{theorem}

\begin{proof}
We divide the proof into four parts. First, consider the natural transformation
\[\alpha\colon \mathrm{id}\rightarrow (k\bBox\mathrm{id})_!\]
induced by the inclusions $[n]\rightarrow I[n]$, dual to the transformation $\beta$ from Proposition~\ref{propBScoloc}. 
Its restriction $\bar{\alpha}\colon y\Rightarrow(k\bBox\mathrm{id})$ in the functor category
$(s\mathbf{S})^{\Delta\times\Delta}$ is a pointwise weak equivalence in $(s\mathbf{S},\mathrm{B})$. That is because 
for every pair $n,m\geq 0$ and every Bousfield-Segal space $X$, the map
$\mathrm{Hom}_2(\alpha([n],[m]),X)$ is isomorphic to the fibration
$((\Delta^n\rightarrow I\Delta^n)\setminus X)^{\Delta^m}$ and hence acyclic by Proposition~\ref{freegrpdinv}. It follows 
that the natural transformation $\alpha\colon y_!\rightarrow (k\bBox\mathrm{id})_!$ between left Kan extensions
is also a pointwise weak equivalence in $(s\mathbf{S},\mathrm{B})$. This can be seen for example using that every 
bisimplicial set is the ``canonical'' homotopy colimit of representables (\cite[Proposition 2.9]{duggerunivhtytheories}) 
and that the left Quillen functor $(k\bBox\mathrm{id})_!$ preserves homotopy colimits.

Second, given a bisimplicial set $X$, let
\[r\colon (k\bBox\mathrm{id})_!(X)\rightarrow X\sprime\]
be a fibrant replacement in $(s\mathbf{S},\mathrm{S})$. Note that by the definition of the left adjoint
$(k\bBox\mathrm{id})_!$, every $f\in (k\bBox\mathrm{id})_!(X)_{1}$ is mapped to a homotopy equivalence in the 
Segal space $X\sprime$, and so the map $r$ factors through $\mathrm{Core}(X\sprime)$.
The canonical inclusion $\iota\colon\mathrm{Core}(X\sprime)\hookrightarrow X\sprime$ is a Reedy fibration between Segal 
spaces by its definition via Diagram (\ref{diaglemmacorestrict1}) and the fact that the inclusion
$X_{\mathrm{hoequiv}}\rightarrow X_1$ is a Kan fibration. As such it is a fibration in $(s\mathbf{S},\mathrm{S})$ as 
well. Hence, we obtain a lift to the square
\[\xymatrix{
(k\bBox\mathrm{id})_!(X)\ar@{^(->}[d]_r\ar[r]^{r}& \mathrm{Core}(X\sprime)\ar@{->>}[d]^{\iota}\\
X\sprime\ar@{=}[r] & X\sprime
}\]
which shows that $\mathrm{Core}(X\sprime)=X\sprime$. This means that $X\sprime$ is a Bousfield-Segal space.

Third, in a similar way to the proof of \cite[Proposition 6.27]{joyalqcatslecture}, let $X$ be a bisimplicial set and 
consider the following commutative diagram.
\[\xymatrix{
X\ar[dr]_{\alpha_X}\ar[r]^(.3){\eta} & (k\bBox\mathrm{id})^!(k\bBox\mathrm{id})_!(X)\ar[rr]^{(k\bBox\mathrm{id})^!(r)}\ar[d]^{\beta_{(k\bBox\mathrm{id})_!(X)}} & & (k\bBox\mathrm{id})^!(X\sprime)\ar[d]^{\beta_{X\sprime}} \\
& (k\bBox\mathrm{id})_!(X)\ar[rr]_{r} & & X\sprime
}\]
To prove that the pair (\ref{BScoloc}) is a homotopy colocalization, we need to show that the top composition of arrows 
is a weak equivalence in $(s\mathbf{S},\mathrm{B})$. In the first step of the proof, we have seen that the map
$\alpha_X$ is a weak equivalence in $(s\mathbf{S},\mathrm{B})$. The same goes for the bottom map $r$ since
$(s\mathbf{S},\mathrm{B})$ is a left Bousfield localization of $(s\mathbf{S},\mathrm{S})$. The vertical map
$\beta_{X\sprime}$ is a Reedy weak equivalence by Proposition~\ref{propBScoloc}, since
$X\sprime=\mathrm{Core}(X\sprime)$ as shown in the second step. Thus, the top horizontal map is a weak equivalence in
$(s\mathbf{S},\mathrm{B})$ by 3-for-2.

Lastly, the fact that $\alpha_X$ is a weak equivalence in $(s\mathbf{S},\mathrm{S})$ whenever $X$ is a Bousfield-Segal 
space follows from a 3-for-2 argument as well. Indeed, the composite
\[X\xrightarrow{\alpha_X}(k\bBox\mathrm{id})_!(X)\xrightarrow{r} X\sprime\]
is a weak equivalence in $(\mathbf{S},\mathrm{B})$ since both of its components are. When $X$ is a Bousfield-Segal 
space, this composite is a weak equivalence between $B$-local objects and hence a weak equivalence in
$(s\mathbf{S},\mathrm{S})$ as well. Since the map $r$ is a weak equivalence in $(s\mathbf{S},\mathrm{S})$ by assumption,  
so is $\alpha_X$. This finishes the proof.
\end{proof}

\section{Complete Bousfield-Segal spaces}\label{chcbs}

In this section we study the model structure for Bousfield-Segal spaces which satisfy Rezk's completeness 
condition. We show that they are exactly the Reedy fibrant bisimplicial sets which are homotopically constant, and hence 
appear in various forms in the existing literature. We will deduce that the model structure $(s\mathbf{S},\mathrm{CB})$ 
for complete Bousfield-Segal spaces is Quillen equivalent to the standard Quillen model structure for Kan complexes.

The model structure $(s\mathbf{S},\mathrm{S})$ for Segal spaces can be localized at the set
\[\mathrm{C}:=\{p_1^{\ast}(\{0\})\colon p_1^{\ast}(\Delta^0)\rightarrow p_1^{\ast}(I\Delta^1)\}.\]
This defines the model structure $(s\mathbf{S},\mathrm{CS}):=\mathcal{L}_{\mathrm{C}}(s\mathbf{S},\mathrm{S})$ which 
originally was presented in \cite{rezk} and is further studied in \cite[Section 4]{jtqcatvsss}. Its fibrant objects are 
the \emph{complete} Segal spaces, i.e.\ the Segal spaces $X$ such that the map
\[\{0\}\setminus X\colon I\Delta^1\setminus X\twoheadrightarrow X_0\]
is an acyclic fibration.
In other words, a Segal space $X$ is complete whenever we may identify vertical homotopies in $X_0$ with horizontal 
equivalences in $X$. This defining identification implies that $(s\mathbf{S},\mathrm{CS})$ is also the localization of
$(s\mathbf{S},\mathrm{S})$ at the class of homotopically fully faithful and essentially surjective maps
(the \emph{Dwyer Kan equivalences}). This is \cite[Corollary 7.9]{rezkhtytps}.
Joyal and Tierney show in \cite{jtqcatvsss} that the homotopy theory of complete Segal spaces is equivalent to the 
homotopy theory of $(\infty,1)$-categories. This justifies to think of complete Segal spaces as the
$(\infty,1)$-category objects in the $(\infty,1)$-category of spaces.

Analogously, localizing $(s\mathbf{S},\mathrm{B})$ at the set C yields the simplicial, left-proper and 
combinatorial model category
\[(s\mathbf{S},\mathrm{CB}):=\mathcal{L}_C(s\mathbf{S},\mathrm{B})=\mathcal{L}_{\mathrm{B}}(s\mathbf{S},\mathrm{CS}).\]

\begin{definition}\label{CBS-spaces}
We say that $X\in s\mathbf{S}$ is a complete Bousfield-Segal space if $X$ is a $B$-local complete Segal space. In other 
words, complete Bousfield-Segal spaces are exactly the fibrant objects in the model category
$(s\mathbf{S},\mathrm{CB})$.
\end{definition}

\begin{remark}
The genealogy of the model structures on $s\mathbf{S}$ that we have considered can be depicted by the following diagram, 
where the arrows indicate the direction of the left Quillen functor.
\[
\xymatrix{
R_v\ar[rr]\ar[d] & & S\ar[d]\ar[dr] & \\
B\ar@{=}[rr]\ar[dr] & & \mathcal{L}_B (S)\ar[dr]& CS\ar[d] \\
 & CB\ar@{=}[rr] & &  \mathcal{L}_B(CS) 
}
\]
The horizontal equality in the middle of the diagram was shown in Theorem~\ref{B=BS}. The bottom equality is an 
immediate corollary. The reader may note however that there are direct and concise proofs of the bottom 
equality which do not require an understanding of the model structure for (non-complete) Bousfield-Segal spaces. For 
example, it can be deduced from the next lemma.
\end{remark}

By \cite[Proposition 2.8]{jtqcatvsss}, every $v$-fibrant bisimplicial set $X$ is \emph{categorically constant} in that 
every map $f\colon [m]\rightarrow [n]$ in $\Delta$ induces a categorical equivalence
$X_{\bullet f}\colon X_{\bullet n}\rightarrow X_{\bullet m}$ between the two rows. In the following we show that a
$v$-fibrant $X$ is a complete Bousfield-Segal space if and only if it is homotopically constant in the vertical 
direction as well.

\begin{definition}
A bisimplicial set $X$ is \emph{homotopically constant} if the map
$X(f)\colon X_{n}\rightarrow X_m$ is a weak homotopy equivalence for every function $f\colon [m]\rightarrow [n]$ in
$\Delta$.
\end{definition}

\begin{lemma}\label{lemmacbs=lc}
A $v$-fibrant bisimplicial set $X$ is a complete Bousfield-Segal space if and only if $X$ is homotopically constant.
\end{lemma}
\begin{proof}
Clearly, $X$ is homotopically constant if and only if all boundary and degeneracy maps of $X$ are weak homotopy
equivalences. This in turn holds if and only if all boundary maps of $X$ are weak homotopy equivalences since 
the degeneracies are sections of the boundaries.

If $X$ is homotopically constant, all the pullbacks $X_1\times_{X_0}^B\dots\times_{X_0}^B X_1$ are contractible over 
$X_0$ via the projection to the vertex $[0]$ (using right properness of $(\mathbf{S},\mathrm{Kan})$). The same holds for 
the $X_n$, and so the Bousfield maps are weak homotopy equivalences over $X_0$ by 3-for-2. It follows that
$(\{0\}\colon\Delta^0\rightarrow I\Delta^1)\setminus X\cong d_1$, 
which is a weak homotopy equivalence by assumption as well. Thus $X$ is complete. 

Vice versa, if $X$ is a complete Bousfield-Segal space, the acyclic fibration
$I\Delta^1\setminus X\twoheadrightarrow X_0$ factors through the fibrations
$I\Delta^1\setminus X\twoheadrightarrow X_1\twoheadrightarrow X_0$. The first of these two is acyclic by
Corollary~\ref{rezktechnical}, and so is the boundary $d_1\colon X_1\twoheadrightarrow X_0$ by 3-for-2. It follows that 
$d_0\colon X_1\twoheadrightarrow X_0$ is a weak homotopy equivalence, too, because $s_0$ is a mutual section. 
Consequently all the pullbacks $X_1\times_{X_0}^B\dots\times_{X_0}^B X_1$ are homotopy equivalent to one another via the 
projections. Since the Bousfield maps are weak homotopy equivalences by assumption, it follows that 
the boundaries of $X$ are weak homotopy equivalences again by a repeated application of 3-for-2.
\end{proof}

\begin{corollary}
The model structure $(s\mathbf{S},\mathrm{CB})$ is cartesian closed, i.e.\ it is a monoidal model category with respect 
to the cartesian product.
\end{corollary}
\begin{proof}
By \cite[Proposition 9.2]{rezk} it suffices to show that the exponential $X^{p_1^{\ast}\Delta^1}$ is a complete 
Bousfield-Segal space whenever $X$ is such. Since $X^{p_1^{\ast}\Delta^1}$ is again Reedy fibrant, by
Lemma \ref{lemmacbs=lc} this means that we have to show that $X^{p_1^{\ast}\Delta^1}$ is homotopically constant. So let 
$d^i\colon [n]\rightarrow [n+1]$ be a coboundary inclusion and let $X$ be a complete Bousfield-Segal space. We want show 
that the boundary $X^{p_1^{\ast}\Delta^1}(d^i)$ is a weak equivalence.
There are natural isomorphisms
\[(X^{p_1^{\ast}(\Delta^1)})_n\cong \Delta^n\setminus X^{p_1^{\ast}(\Delta^1)}\cong(\Delta^1\times\Delta^n)\setminus X\]
and $X^{p_1^{\ast}(\Delta^1)}(d^i)\cong (\mathrm{id}_{\Delta^1}\times d^i)\setminus X$.
By assumption $X$ is homotopically constant and it follows that $\blank\setminus X$ takes all anodyne maps to trivial 
fibrations by \cite[Lemma 3.7]{jtqcatvsss}.  Thus, $(\mathrm{id}_{\Delta^1}\times d^i)\setminus X$ is an acyclic 
fibration.
\end{proof}

From Lemma~\ref{lemmacbs=lc} it follows that the model structure $(s\mathbf{S},\mathrm{CB})$ is contained in various 
classes of well understood model structures studied in the literature. 

In \cite{rss} for instance, given a model category $\mathbb{M}$, the model structure on $\mathbb{M}^{\Delta^{op}}$ 
whose fibrant objects are exactly the homotopically constant Reedy fibrant simplicial objects is called the 
\emph{canonical model structure} on $\mathbb{M}^{{\Delta}^{op}}$ (whenever it exists). In the case
$\mathbb{M}=(\mathbf{S},\mathrm{Kan})$, this implies that the projection
$\iota_2^{\ast}\colon(s\mathbf{S},\mathrm{CB})\rightarrow(\mathbf{S},\mathrm{Kan})$ 
onto the first column is part of a Quillen equivalence by \cite[Theorem 3.6]{rss}. The fact that the projection
$\iota_2^{\ast}$ is part of a Quillen equivalence was stated in Bergner's paper (\cite[Theorem 6.12]{bergner2}) as well, 
albeit its proof was omitted. 

By a symmetry inherent to the model structure $(s\mathbf{S},\mathrm{CB})$, this implies that the ``perpendicular'' 
projection $\iota_1^{\ast}\colon(s\mathbf{S},\mathrm{CB})\rightarrow(\mathbf{S},\mathrm{Kan})$ onto the first row is a 
Quillen equivalence as well, as the following theorem shows.

\begin{theorem}\label{cnstequiv2}
The pair
\[(p_1^{\ast},\iota_1^{\ast})\colon(\mathbf{S},\mathrm{Kan})\rightarrow(s\mathbf{S},\mathrm{CB})\]
is a Quillen equivalence.
\end{theorem}

\begin{proof}
Recall the involution $\sigma^{\ast}\colon s\mathbf{S}\rightarrow s\mathbf{S}$ 
induced by the permutation $\sigma\colon\Delta\times\Delta\rightarrow\Delta\times\Delta$ which swaps the components
$([n],[m])\mapsto([m],[n])$. Following the notation from Section~\ref{secreedy}, note that
$\sigma^{\ast}[\mathcal{W}_v]=\mathcal{W}_h$, $\sigma^{\ast}[\mathcal{C}_v]=\mathcal{C}_h=\mathcal{C}$ and even
$\sigma^{\ast}[\mathcal{I}_v]=\mathcal{I}_h$ and $\sigma^{\ast}[\mathcal{J}_v]=\mathcal{J}_h$ since
$\sigma^{\ast}$ preserves colimits. Furthermore, for all objects $X,Y\in s\mathbf{S}$ the involution satisfies the 
following identities.
\begin{align}\label{equcnstequiv2}
\mathrm{Hom}_2(\sigma^{\ast}X,\sigma^{\ast}Y):=\iota_2^{\ast}((\sigma^{\ast}Y)^{\sigma^{\ast}X})=\iota_2^{\ast}\sigma^{\ast}(Y^X)=\iota_1^{\ast}(Y^X)=:\mathrm{Hom}_1(X,Y)
\end{align}
In analogy to Proposition~\ref{prophom2}, the functor $\mathrm{Hom}_1$ turns $(s\mathbf{S},R_h)$ into a simplicial 
model category. Let
\[CB^{\perp}:=\{p_2^{\ast}(\{0\})\}\cup\{p_2^{\ast}(c_{n})\mid n\geq 2\},\] 
so we can construct the Bousfield localization $\mathcal{L}_{CB^{\perp}}(s\mathbf{S},R_h)$. Via 
(\ref{equcnstequiv2}) one computes that a bisimplicial set $X$ is $h$-fibrant and $\mathrm{CB}^{\perp}$-local
(with respect to the $\mathrm{Hom}_1$-enrichment) if and only if it is $v$-fibrant and $\mathrm{C\cup B}$-local (with 
respect to the $\mathrm{Hom}_2$-enrichment). 
Then the model structures $(s\mathbf{S},\mathrm{CB})=\mathcal{L}_{C\cup B}(s\mathbf{S},\mathrm{R}_v)$ and
$\mathcal{L}_{\mathrm{CB}^{\perp}}(s\mathbf{S},R_h)$ coincide since they have the same cofibrations and the same fibrant 
objects.
Hence, the involution $(\sigma^{\ast},\sigma^{\ast})$ is a Quillen equivalence from $(s\mathbf{S},\mathrm{CB})$ to 
itself. By \cite[Theorem 3.6]{rss}, the projection
$\iota_2^{\ast}\colon(s\mathbf{S},\mathrm{CB})\rightarrow(\mathbf{S},\mathrm{Kan})$ onto the first column is a Quillen
equivalence, and so it follows that the first row projection $\iota_1^{\ast}=\iota_2^{\ast}\circ\sigma^{\ast}$ 
is part of a Quillen equivalence as well.
\end{proof}

One also can show Theorem~\ref{cnstequiv2} in another way, using that the pair
$(p_1^{\ast},\iota_1^{\ast})$ gives rise to a Quillen equivalence between $(\mathbf{S},\mathrm{QCat})$ and
$(s\mathbf{S},\mathrm{S})$  as shown in \cite[Theorem 4.11]{jtqcatvsss}. Localizing both sides at the left horn 
inclusions yields the same result (see \cite[Theorem 5.1.14]{thesis}).

\begin{theorem}\label{diagcbs1}
The diagonal $d^{\ast}\colon s\mathbf{S}\rightarrow\mathbf{S}$ is part of a Quillen equivalence
\[(d^{\ast},d_{\ast})\colon(s\mathbf{S},\mathrm{CB})\rightarrow(\mathbf{S},\mathrm{Kan}).\]
\end{theorem}

\begin{proof}
The statement can be shown along the lines of \cite[Theorem 4.12]{jtqcatvsss}, using Theorem~\ref{cnstequiv2} 
and a 3-for-2 argument. 
\end{proof}

The fact that the diagonal induces an equivalence on the homotopy categories of the two model structures is 
exactly the unpointed version of \cite[Theorem 3.1]{bousfieldnotes} for very special bisimplicial sets of type $n=0$.

\begin{remark}\label{rembshfib}
We have seen in the proof of Theorem~\ref{cnstequiv2} that $(s\mathbf{S},\mathrm{CB})$ is also a left Bousfield 
localization of the horizontal Reedy structure $(s\mathbf{S},\mathrm{R}_h)$. It thus follows that every complete 
Bousfield-Segal space is both $v$-fibrant and $h$-fibrant. Vice versa, every $h$-fibrant bisimplicial set is 
homotopically constant by \cite[Proposition 2.8]{jtqcatvsss} (or, more precisely, by its horizontal dual). Thus, by 
Lemma~\ref{lemmacbs=lc} it follows that every bisimplicial set which is simultaneously $v$-fibrant and $h$-fibrant is a 
complete Bousfield-Segal space. In this sense, the property of being a complete Bousfield-Segal space is the conjunction 
of two Reedy fibrancy conditions. We will use this in Theorem~\ref{rightproper} to give a direct proof of right 
properness of $(s\mathbf{S},\mathrm{CB})$.
\end{remark}

The model structure $(s\mathbf{S},\mathrm{CB})$ furthermore coincides with Dugger's \emph{realization} or \emph{hocolim} 
model structure on $s\mathbf{S}$ (\cite{duggersimp}) and with Cisinski's model structure for ``locally constant 
presheaves'' on $\Delta$ (\cite{cisinski}). This implies that $(s\mathbf{S},\mathrm{CB})$ yields a model for 
univalent type theory as discussed in the next section.

\section{Right properness and semantics of univalent type theory}\label{seccbshott}

In this last section, we prove right properness of $(s\mathbf{S},\mathrm{CB})$ and record that the 
model structure furthermore is ``type theoretic'' in the sense that it is a model for Homotopy Type Theory with 
univalent universes as treated in the HoTT-Book \cite{hott}.
 
Consequently, complete Bousfield-Segal spaces and their fibrations can be studied formally in the syntax of their 
underlying univalent type theory. By soundness of the calculus, all results proven synthetically in this type theory 
automatically follow to hold in their associated model category $(s\mathbf{S},\mathrm{CB})$ as well.

The interpretation of HoTT in $(s\mathbf{S},\mathrm{CB})$ can be constructed from a combination of existing results in 
the literature. Indeed, we can take the results of Shulman in \cite{shulman1} as a given to reduce the construction of a 
semantics for HoTT to the verification of a few model categorical properties. 
More precisely, the model structure $\mathrm{CB}$ is a cofibrantly generated model structure on the presheaf category
$s\mathbf{S}$ whose cofibrations are exactly the monomorphisms. This means that the simplicial model category
$(s\mathbf{S},\mathrm{CB})$ defines a \emph{Cisinski model category}. Therefore, by \cite[Theorem 5.1]{shulman1}, 
in order for $(s\mathbf{S},\mathrm{CB})$ to support a univalent type theoretical interpretation, we only have to 
show that $(s\mathbf{S},\mathrm{CB})$ is right proper and that it comes equipped with an infinite sequence of univalent 
universal fibrations. 

On the one hand, both properties follow from abstract results of Cisinski (Observation 
1). On the other hand, one obtains the universal fibrations by more concrete means via a result of Rezk, 
Schwede and Shipley (Observation 2). The fact that $(s\mathbf{S},\mathrm{CB})$ is right proper can be proven directly by 
results from the previous section (Observation 3).

\paragraph*{Observation 1}\label{hottpara1}
In \cite[1]{cisinski}, Cisinski introduces the \emph{locally constant model structure}
$([\mathcal{A}^{op},\mathbf{S}],\mathrm{lc})$ on simplicial presheaves over any elegant Reedy category $\mathcal{A}$. It 
is a left Bousfield localization of the Reedy model structure, its fibrant objects are exactly the homotopically 
constant Reedy fibrant objects $X\in[\mathcal{A}^{op},\mathbf{S}]$. These are the Reedy fibrant objects $X$ such that 
the $f$-action $X(f)\colon X(b)\rightarrow X(a)$ is a weak homotopy equivalence for all maps $f\colon a\rightarrow b$ in
$\mathcal{A}$. Hence, Lemma~\ref{lemmacbs=lc} shows that
$(s\mathbf{S},\mathrm{lc})=(s\mathbf{S},\mathrm{CB})$. In \cite{cisinskiprefaisc}, he shows
that $([\mathcal{A}^{op},\mathbf{S}],\mathrm{lc})$ is always right proper, drawing back on general observations about 
fundamental localizers. In \cite[Proposition 1.1]{cisinski} he shows that this model category contains a fibrant 
univalent universe classifying $\kappa$-small maps for every inaccessible cardinal $\kappa$ large enough.

\paragraph*{Observation 2}
More concretely, we obtain a sequence of univalent universes for $(s\mathbf{S},\mathrm{CB})$ from
\cite[Theorem 3.2]{shulman1} and \cite[Corollary 2.5.9]{thesis} if we can simply show that there is a set of generating 
acyclic cofibrations for $(s\mathbf{S},\mathrm{CB})$ with representable codomain. Since we obtained the model structure
$(s\mathbf{S},\mathrm{CB})$ by left Bousfield localization, a priori it is very hard to present a well behaved set of 
generating acyclic cofibrations. But the authors of \cite{rss} show that the fibrations in the canonical model structure
$(s\mathbf{S},\mathrm{CB})$ are exactly the \emph{equi-fibred Reedy fibrations}. For such, a set of generating acyclic 
cofibrations is given in \cite[Proposition 8.5]{rss} by $\mathcal{J}_{CB}=\mathcal{J}_h\cup\mathcal{J}^{\prime\prime}$ 
for
\[\mathcal{J}_h=\{h_i^n\Box\sprime\delta_m\colon(\Delta^n\Box\partial\Delta^m)\cup_{\Lambda_i^n\Box\partial\Delta^m}(\Lambda_i^n\Box\Delta^m)\rightarrow(\Delta^n\Box\Delta^m)\mid 0\leq i\leq m,n\},\]
\[\mathcal{J}^{\prime\prime}:=\{\delta_n\Box\sprime d^m_i\colon(\Delta^n\Box\Delta^{m-1})\cup_{\partial\Delta^n\Box\Delta^{m-1}}(\partial\Delta^n\Box\Delta^m)\rightarrow(\Delta^n\Box\Delta^m)\mid n\geq 0,m\geq i\geq 0\}.\]
The box products $\Delta^n\Box\Delta^m$ are exactly the representables in $s\mathbf{S}$, thus a set of 
generating acyclic cofibrations with representable codomain does indeed exist.

\paragraph*{Observation 3}
We can give an elementary proof of right properness of $(s\mathbf{S},\mathrm{CB})$ using the symmetry of the 
model structure observed in the last section. More precisely, we can use that complete Bousfield-Segal spaces are 
exactly the objects which are simultaneously $v$-fibrant and $h$-fibrant as noted in Remark~\ref{rembshfib}.

Recall that a model category $\mathbb{M}$ is right proper if and only if the pullback of any acyclic cofibration with 
fibrant codomain along fibrations is a weak equivalence. This is shown in \cite[Lemma 9.4]{bousfieldtelescopic} for 
example. Furthermore, recall the generating sets $\mathcal{J}_v$ and $\mathcal{J}_h$ for the acyclic cofibrations in
$(s\mathbf{S},R_v)$ and in $(s\mathbf{S},R_h)$, respectively, from Section~\ref{secreedy}. In the following, we denote 
the class of weak equivalences in $(s\mathbf{S},\mathrm{CB})$ by $\mathcal{W}_{CB}$, and the class of its 
fibrations by $\mathcal{F}_{CB}$. Given a class $S$ of arrows in $s\mathbf{S}$, we denote the class of arrows with the 
left lifting property (right lifting property) against all arrows in $S$ by $\tensor[^\pitchfork]{S}{}$ (by $S^{\pitchfork}$).

\begin{lemma}\label{rpreduction2}
The class of acyclic cofibrations with fibrant codomain in $(s\mathbf{S},\mathrm{CB})$ is exactly the class of 
maps in the saturation of $\mathcal{J}_v\cup\mathcal{J}_h$ with fibrant codomain, i.e.\
\[(\mathcal{W}_{CB}\cap\mathcal{C})/\text{C.\ B.-S.\ spaces}=\tensor[^\pitchfork]{((\mathcal{J}_v\cup\mathcal{J}_h)}{^\pitchfork})/\text{C.\ B.-S.\ spaces}.\]
\end{lemma}
\begin{proof}
As $(s\mathbf{S},\mathrm{CB})$ is a left Bousfield localization of both $(s\mathbf{S},R_v)$ and 
$(s\mathbf{S},R_h)$, we have
\[\mathcal{J}_v\cup\mathcal{J}_h\subseteq\mathcal{W}_{CB}\cap\mathcal{C},\]
so one direction is clear. Vice versa, let $j\colon X\hookrightarrow Y$ be a weak equivalence in
$(s\mathbf{S},\mathrm{CB})$ with $Y$ a complete Bousfield-Segal space. Note that
$(\mathcal{J}_v\cup\mathcal{J}_h)^{\pitchfork}$ is the intersection of the set $\mathcal{F}_v$ of $v$-fibrations and the 
set $\mathcal{F}_h$ of $h$-fibrations, and hence the pair\linebreak
$(\tensor[^\pitchfork]{((\mathcal{J}_v\cup\mathcal{J}_h)}{^\pitchfork}),\mathcal{F}_v\cap\mathcal{F}_h)$ is a 
weak factorization system on $s\mathbf{S}$ by general abstract non-sense. Pick a factorization
$X\xrightarrow{k}Z\xrightarrow{q}Y$ of $j$ with
$k\in\tensor[^\pitchfork]{((\mathcal{J}_v\cup\mathcal{J}_h)}{^\pitchfork})$ and
$q\in\mathcal{F}_v\cap\mathcal{F}_h$,
\begin{equation}\label{diagrpreduction2}
\begin{gathered}
\xymatrix{
X\ar[r]^k\ar@{^(->}[d]^j & Z\ar[d]^{q} \\
Y\ar@{=}[r] & Y. \\
}
\end{gathered}
\end{equation}
Since $Y$ is a complete Bousfield-Segal space, $Z$ is now both $v$-fibrant and $h$-fibrant, hence a complete
Bousfield-Segal space, too. But a map between complete Bousfield-Segal spaces is a fibration in
$(s\mathbf{S},\mathrm{CB})$ if and only if it is a $v$-fibration. This in turn holds if and only if it is an
$h$-fibration as can be seen by \cite[Proposition 7.21]{jtqcatvsss}. Hence, we obtain a lift for the square
$(\ref{diagrpreduction2})$ which exhibits $j$ as retract of $k$. Therefore,
$j\in\tensor[^\pitchfork]{((\mathcal{J}_v\cup\mathcal{J}_h)}{^\pitchfork})$.
\end{proof}

\begin{lemma}\label{rpreduction3}
Let $S$ and $\mathcal{F}$ be classes of morphisms in $s\mathbf{S}$ and assume $\mathcal{F}$ is closed under pullbacks, 
retracts and sequential limits. Suppose for every map $f\colon X\rightarrow Y$ in $S$ and every 
map $p\colon Z\rightarrow Y$ in $\mathcal{F}$ the pullback $p^{\ast}f\colon p^{\ast}X\rightarrow Z$ is in the 
saturation $\tensor[^\pitchfork]{((S)}{^\pitchfork})$ of $S$. Then the pullback of every map in the saturation
$\tensor[^\pitchfork]{((S)}{^\pitchfork})$ along a map in $\mathcal{F}$ is again contained in
$\tensor[^\pitchfork]{((S)}{^\pitchfork})$ as well.
\end{lemma}
\begin{proof}
In the language of \cite[3]{shulman1}, this holds in virtue of the ``exactness'' properties of Grothendieck 
toposes, i.e.\ pullbacks in $s\mathbf{S}$ commute with pushouts, transfinite compositions and retracts in such 
a way that the proof becomes a straight forward induction.
\end{proof}
\begin{theorem}\label{rightproper}
The model category $(s\mathbf{S},\mathrm{CB})$ is right proper.
\end{theorem}
\begin{proof}
The class $\mathcal{F}_{\mathrm{CB}}$ is closed under pullbacks, retracts and sequential limits. Hence, by
Lemma~\ref{rpreduction2} and Lemma~\ref{rpreduction3} it remains to check that a pullback square of the form
\[
\xymatrix
{
P\ar[r]\ar[d]_{p^{\ast}j}\ar@{}[dr]|(.3){\pbs} & Y\ar[d]^{j} \\
X\ar@{->>}[r]_(.4)p & \Delta^n\Box\Delta^m \\
}
\]
with a fibration $p$ in $(s\mathbf{S},\mathrm{CB})$ and $j\in \mathcal{J}_v\cup \mathcal{J}_h$ exhibits the arrow 
$p^{\ast}j$ to be a weak equivalence in $(s\mathbf{S},\mathrm{CB})$. But $\mathcal{F}_{CB}$ is a subset of
$\mathcal{F}_v\cap\mathcal{F}_h$, so $p$ is both a $v$-fibration and an $h$-fibration. Both Reedy structures
$(s\mathbf{S},R_v)$ and $(s\mathbf{S},R_h)$ are right proper due to the right properness of
$(\mathbf{S},\mathrm{Kan})$. Therefore, $p^{\ast}j\in\mathcal{W}_v\cup\mathcal{W}_h$. But both $\mathcal{W}_v$ 
and $\mathcal{W}_h$ are contained in $\mathcal{W}_{CB}$, since the model structure $CB$ is a left Bousfield 
localization of both. This finishes the proof.
\end{proof}

In \cite[Lemma 7.3.9]{thesis} it is shown that the left Bousfield localization of a right proper model category is
right proper again if and only if the localization is semi-left exact (\cite[Definition 7.3.1]{thesis}). That means, if 
the left derived of the localization preserves homotopy pullbacks along maps between local objects. It  therefore
follows from Theorem~\ref{rightproper} that the localization
$(s\mathbf{S},\mathrm{R}_v)\rightarrow(s\mathbf{S},\mathrm{CB})$ is semi-left exact. 

Furthermore, in virtue of the Quillen equivalence to $(\mathbf{S},\mathrm{Kan})$, the model category
$(s\mathbf{S},\mathrm{CB})$ is a model topos in the sense of \cite{rezkhtytps}. One may therefore ask whether the
semi-left exact localization $(s\mathbf{S},\mathrm{R}_v)\rightarrow(s\mathbf{S},\mathrm{CB})$ is in fact left 
exact. Or in other words, whether the model topos $(s\mathbf{S},\mathrm{CB})$ is a subtopos of the presheaf model topos
$(s\mathbf{S},\mathrm{R}_v)$. 
Therefore, recall that a left Bousfield localization is left exact if the left derived of the localization preserves all 
homotopy pullbacks.

\begin{proposition}
The localization $(s\mathbf{S},R_v)\rightarrow(s\mathbf{S},\mathrm{CB})$ is not left exact.
\end{proposition}
\begin{proof}
Since every map between non-empty (discrete simplicial) sets is a Kan fibration, every map $S\rightarrow T$ of 
simplicial sets induces a Reedy fibration $p_1^{\ast}S\rightarrow p_1^{\ast}T$ of bisimplicial sets. Let
\[\xymatrix{
P\ar[d]\ar[r]\ar@{}[dr]|(.2){\pbs} & C\ar[d]^{} \\
A\ar[r] & B
}\]
be a cartesian square in $\mathbf{S}$ such that $C\rightarrow B$ is a weak homotopy equivalence and its pullback
$P\rightarrow A$ is not. Then
\[\xymatrix{
p_1^{\ast}P\ar[d]\ar[r]\ar@{}[dr]|(.2){\pbs} & p_1^{\ast}C\ar[d]^{} \\
p_1^{\ast}A\ar@{->>}[r] & p_1^{\ast}B
}\]
is cartesian in $s\mathbf{S}$, $p_1^{\ast}C\rightarrow p_1^{\ast}B$ is a weak equivalence in
$(s\mathbf{S},\mathrm{CB})$ and $p_1^{\ast}A\rightarrow p_1^{\ast}B$ is a Reedy fibration (although
$A\rightarrow B$ is not a Kan fibration). Then $p_1^{\ast}P\rightarrow p_1^{\ast}A$ cannot be a weak 
equivalence in $(s\mathbf{S},\mathrm{CB})$, because $p_1^{\ast}$ is the left adjoint of a
Quillen equivalence and hence reflects weak equivalences between cofibrant objects. In particular, the square 
cannot be homotopy cartesian in $(s\mathbf{S},\mathrm{CB})$. But it is homotopy cartesian in
$(s\mathbf{S},R_v)$, because $p_1^{\ast}A\rightarrow p_1^{\ast}B$ is a Reedy fibration and the Reedy model structure is 
right proper.
\end{proof}

\bibliographystyle{amsplain}
\bibliography{BSBib}

\newcommand{\noopsort}[1]{}
\providecommand{\bysame}{\leavevmode\hbox to3em{\hrulefill}\thinspace}
\providecommand{\MR}{\relax\ifhmode\unskip\space\fi MR }
\providecommand{\MRhref}[2]{%
  \href{http://www.ams.org/mathscinet-getitem?mr=#1}{#2}
}
\providecommand{\href}[2]{#2}
\begin{thebibliography}{10}

\bibitem{bergner2}
J.E. Bergner, \emph{Adding inverses to diagrams {II}: {I}nvertible homotopy
  theories are spaces}, Homology, Homotopy and Applications \textbf{10} (2008),
  no.~2, 175–193.

\bibitem{brtheta}
J.E. Bergner and C.~Rezk, \emph{Reedy categories and the
  $\theta$-construction}, Mathematische Zeitschrift \textbf{274} (2013),
  no.~1-2, 499–514.

\bibitem{bousfieldnotes}
A.K. Bousfield, \emph{The simplicial homotopy theory of iterated loop spaces},
  Typed notes by Julie Bergner.

\bibitem{bousfieldtelescopic}
\bysame, \emph{On the telescopic homotopy theory of spaces}, Transactions of
  the American Mathematical Society \textbf{353} (2001), no.~6, 2391--2426.

\bibitem{cisinskiprefaisc}
D.C. Cisinski, \emph{Les pr\'{e}faisceaux comme mod\`{e}les des types
  d’homotopie}, Ast\'{e}risque \textbf{308} (2007), 392 pp.

\bibitem{cisinski}
\bysame, \emph{Univalent universes for elegant models of homotopy types},
  \url{http://arxiv.org/abs/1406.0058}, 2014, [Online, accessed 31 May 2014].

\bibitem{duggersmallpres}
D.~Dugger, \emph{Combinatorial model categories have presentations}, Adv. Math.
  \textbf{164} (2001), no.~1, 177--201.

\bibitem{duggersimp}
\bysame, \emph{Replacing model categories with simplicial ones}, Transactions
  of the American Mathematical Society \textbf{353} (2001), no.~12, 5003--5027.

\bibitem{duggerunivhtytheories}
\bysame, \emph{Universal homotopy theories}, Advances in Mathematics
  \textbf{164} (2001), no.~1, 144--176.

\bibitem{hirschhorn03}
P.S. Hirschhorn, \emph{Model categories and their localizations}, Mathematical
  Surveys and Monographs, no.~99, American Mathematical Society, Providence,
  R.I., 2003.

\bibitem{hovey}
M.~Hovey, \emph{Model categories}, Mathematical Surveys and Monographs,
  vol.~63, American Mathematical Society, 1999.

\bibitem{joyalqcatslecture}
A.~Joyal, \emph{The {T}heory of {Q}uasi-{C}ategories and its {A}pplications},
  \url{http://mat.uab.cat/~kock/crm/hocat/advanced-course/Quadern45-2.pdf},
  2008-2009, [Lecture Series at ``Simplicial Methods in Higher Categories''
  (CRM)].

\bibitem{jtqcatvsss}
A.\ Joyal and M.\ Tierney, \emph{Quasi-categories vs {S}egal spaces},
  Categories in {A}lgebra, {G}eometry and {M}athematical {P}hysics, American
  Mathematical Society, 2006, pp.~277--326.

\bibitem{hallgrpthy}
M.~Hall Jr., \emph{The theory of groups}, The Macmillan Company, New York, 1959
  (Fourth Printing 1963).

\bibitem{luriehtt}
J.~Lurie, \emph{Higher topos theory}, Annals of Mathematics Studies, no. 170,
  Princeton University Press, 2009.

\bibitem{moerdijkgroupcompletion}
I.~Moerdijk, \emph{Bisimplicial sets and the group completion theorem},
  {A}lgebraic K-Theory: {C}onnections with {G}eometry and {T}opology,
  Mathematical and Physical Sciences, vol. 279, Kluwer Academic Publishers,
  1989, pp.~225--240.

\bibitem{hott}
The Univalent~Foundations Program, \emph{Homotopy type theory: Univalent
  foundations of mathematics}, \url{http://homotopytypetheory.org/book}, 2013.

\bibitem{rezk}
C.~Rezk, \emph{A model for the homotopy theory of homotopy theories},
  Transactions of the American Mathematical Society (1999), 973--1007.

\bibitem{rezkhtytps}
C.~Rezk, \emph{Toposes and homotopy toposes (version 0.15)},
  \url{https://www.researchgate.net/publication/255654755_Toposes_and_homotopy_toposes_version_015},
  2010.

\bibitem{rss}
C.~Rezk, S.~Schwede, and B.~Shipley, \emph{Simplicial structures on model
  categories and functors}, American Journal of Mathematics \textbf{123}
  (2001), 551--575.

\bibitem{shulman1}
M.~Shulman, \emph{The {U}nivalence axiom for elegant {R}eedy presheaves},
  Homology, {H}omotopy and {A}pplications \textbf{17} (2013), no.~2.

\bibitem{shulmaninv}
\bysame, \emph{{U}nivalence for inverse diagrams and homotopy canonicity},
  Mathematical Structures in Computer Science, From type theory and homotopy
  theory to Univalent Foundations of Mathematics, vol. 25, Special Issue 5,
  Cambridge University Press, 2015, pp.~1203--1277.

\bibitem{thesis}
R.~Stenzel, \emph{On univalence, {R}ezk completeness and presentable
  quasi-categories}, Ph.D. thesis, University of Leeds, Leeds LS2 9JT, 3 2019.

\end{thebibliography}
\end{document}